\newcommand\bbR{\mathbb{R}}
\newcommand\bbS{\mathbb{R}^{d\times d}_{\text{sym}}}

\newcommand\cC{\mathcal{C}}
\newcommand\cF{\mathcal{F}}
\newcommand\cH{\mathcal{H}}
\newcommand\cP{\mathcal{P}}
\newcommand\cT{\mathcal{T}}

\newcommand\cX{\mathcal{X}}

\newcommand\bn{\boldsymbol{n}}
\newcommand\bF{\boldsymbol{f}}
\newcommand\bu{\boldsymbol{u}}
\newcommand\bv{\boldsymbol{v}}
\newcommand\bw{\boldsymbol{w}}
\newcommand\bd{\boldsymbol{d}}
\newcommand\be{\boldsymbol{e}}

\newcommand\beps{\boldsymbol{\varepsilon}}
\newcommand\bsig{\boldsymbol{\sigma}}
\newcommand\btau{\boldsymbol{\tau}}
\newcommand\bnabla{\boldsymbol{\nabla}}

\newcommand\bchi{\boldsymbol{\chi}}

\newcommand\ubu{\underline{\boldsymbol{u}}}
\newcommand\ubv{\underline{\boldsymbol{v}}}

\newcommand\mt{\mathtt{t}}

\newcommand{\tr}{\operatorname{tr}}
\renewcommand{\div}{\operatorname{div}}
\newcommand{\bdiv}{\operatorname{\mathbf{div}}}

\documentclass[11pt, leqno]{article} 
\usepackage[utf8]{inputenc} 
\usepackage[T1]{fontenc}
\usepackage[english]{babel} 

\usepackage{lmodern} 

\usepackage{amsmath,amsthm,amsfonts,amssymb}

\usepackage{ninecolors}
\usepackage{graphicx}
\usepackage{tabularray}
\usepackage{caption, subcaption}
\usepackage{tikz}

\usepackage{xfrac} 

\usepackage{stmaryrd}

\newcommand{\mmean}[1]{\left\{\kern-1.ex\left\{ #1 \right\}\kern-1.ex\right\}}

\usepackage{mathtools}
\DeclarePairedDelimiterX\norm[1]\lVert\rVert{
   \ifblank{#1}{\:\cdot\:}{#1}
}
\DeclarePairedDelimiterX\abs[1]\lvert\rvert{
   \ifblank{#1}{\:\cdot\:}{#1}
}
\DeclarePairedDelimiter{\inner}{(}{)}
\DeclarePairedDelimiter{\set}{\{}{\}}
\DeclarePairedDelimiter{\dual}{\langle}{\rangle}
\DeclareMathOperator*{\esssup}{ess\,sup}

\usepackage{siunitx}

\usepackage[singlespacing]{setspace} 

\usepackage{titlesec}
\titleformat{\section}[block]{\filcenter\bfseries}{\thesection.}{1em}{}
\titleformat{\subsection}[block]{\bfseries}{\thesubsection.}{1em}{}

\newtheoremstyle{paper}{}{}{\itshape}{}{\scshape}{.}{.5em}{}
\theoremstyle{paper}
\newtheorem{theorem}{Theorem}
\newtheorem{proposition}{Proposition}
\newtheorem{lemma}{Lemma}

\newtheorem{remark}{Remark}

\usepackage[numbers]{natbib}
\usepackage{natbib}
\setcitestyle{aysep={}}

\AtBeginEnvironment{thebibliography}{\setstretch{1.1}}

\usepackage{fancyhdr} 

\usepackage[pdftex, hidelinks, hypertexnames = false, hyperfootnotes = false,
pdfpagemode = UseNone, pdfdisplaydoctitle = true]{hyperref}%

\usepackage[a4paper, tmargin=3cm, bmargin=3cm, rmargin=2.2cm, lmargin=2.2cm]{geometry}

\begin{document}
\title{\huge An $hp$ Error Analysis of HDG for Dynamic Poroelasticity}
\author{Salim Meddahi
	\thanks{This research was  supported by Ministerio de Ciencia e Innovación, Spain.}}



\maketitle

\begin{abstract}

	\noindent
	This study introduces a hybridizable discontinuous Galerkin (HDG) method for simulating low-frequency wave propagation in poroelastic media. We present a novel four-field variational formulation and establish its well-posedness and energy stability. Our \(hp\)-convergence analysis of the HDG method for spatial discretization is complemented by a Crank–Nicolson scheme for temporal discretization. Numerical experiments validate the theoretical convergence rates and demonstrate the effectiveness of the method in accurately capturing poroelastic dynamics.
\end{abstract}
\bigskip

\noindent
\textbf{Mathematics Subject Classification.}  65N12, 65N15, 65N30
\bigskip

\noindent
\textbf{Keywords.}
Poroelasticity, hybridizable discontinuous Galerkin, $hp$ error estimates


\section{Introduction}\label{s:introduction}

The theory of poroelasticity, pioneered by Biot in 1941 and further developed in 1955 \cite{Biot1956}, establishes a mathematical framework to describe the coupled interactions between fluid flow and elastic deformation in porous media. This theory has found diverse applications across multiple disciplines. In geomechanics, it enables the analysis of soil consolidation, reservoir compaction, and land subsidence \cite{gai2004}. In biomedical science, it supports modeling the mechanical behavior of soft tissues, including brain matter and cartilage \cite{quarteroni2023}. The framework is also essential in various industrial applications, particularly in oil and gas extraction, carbon dioxide sequestration, and geothermal energy production.

The literature predominantly focuses on quasi-static poroelastic formulations, in which acceleration effects are negligible. Significant advances have been made in developing robust discretization schemes for Biot's consolidation problems \cite{Boffi2016, Chen2013, Fu2019, Hu2017, Lederer2021, Lee2023c, Oyarzua2016}. However, many practical applications require Biot's dynamic model \cite{Biot1956}, which captures wave propagation phenomena by incorporating acceleration effects in both the solid and fluid phases. These dynamic effects are essential for understanding wave propagation in poroelastic media \cite{morency2010}, and addressing this challenge has led to the development of various computational approaches, including finite difference, finite element, boundary element, finite volume, and spectral methods \cite{gaspar2003, santos1986II, Lemoine2013, chen1995, morency}. More recent approaches have considered high-order space-time continuous and discontinuous Galerkin (DG) methods \cite{antoniettiIMA, bause2024} and mixed finite element methods \cite{Lee2023}.

In this work, we propose a hybridizable discontinuous Galerkin (HDG) method \cite{Cockburn2009} for the low-frequency Biot system. Like DG methods, HDG supports \(hp\)-adaptivity and flexible mesh designs while requiring fewer global degrees of freedom, a significant advantage for computationally demanding problems. Although HDG methods are well established for porous media and elastodynamics \cite{nguyen2011, duSayas2020, meddahi2023hp}, only Hungria \cite{hungria2019} has previously addressed the fully dynamic Biot poroelasticity system. Our approach employs the same variables (fluid pressure, stress tensor, and solid/fluid velocities) but utilizes different numerical traces, with polynomial approximations of degree \(k \geq 0\) for stress and pressure and degree \(k+1\) for velocities and numerical traces. This parameter-free HDG method is applicable in both two and three dimensions, achieving quasi-optimal convergence with respect to mesh size and a sub-optimal rate (off by half a power) with respect to the polynomial degree. We also show that a fully discrete Crank–Nicolson scheme maintains stability and convergence. Finally, we perform a series of numerical experiments that confirm the effectiveness and robustness of our method.

The structure of this paper is as follows. We begin by introducing notation and definitions related to functional spaces. In Section~\ref{sec:model}, we present the linear dynamic poroelastic model problem and its weak formulation in terms of pressure, stress, and fluid and solid velocities. Section~\ref{sec:wellposedness} establishes the existence and uniqueness of the solution. Section~\ref{sec:FE} discusses essential \(hp\) technical requirements necessary for our analysis. The semidiscrete hybridizable discontinuous Galerkin method is introduced in Section~\ref{sec:semi-discrete}, where we also demonstrate its well-posedness. A comprehensive \(hp\) convergence analysis of the HDG method is provided in Section~\ref{sec:convergence}. Section~\ref{sec:fully-discrete} covers the fully discrete scheme. Finally, Section~\ref{sec:numresults} presents numerical results that corroborate the expected convergence rates and applies the proposed method to a geophysical test benchmark.

\subsection{Notations and Sobolev spaces}

For any $m, n \in \mathbb{N}$, we denote by $\mathbb{R}^{m\times n}$ the space of real $m \times n$ matrices. To maintain consistent notation, we identify $\mathbb{R}^{m\times 1}$ with the space of column vectors $\mathbb{R}^{m}$ and $\mathbb{R}^{1\times 1}$ with the scalar field $\mathbb{R}$. For $m > 1$, let $\mathrm{I}_m \in \mathbb{R}^{m \times m}$ be the identity matrix and consider the subspace $\mathbb{R}^{m \times m}_{\text{sym}} \coloneq \{\btau \in \mathbb{R}^{m \times m} \mid \btau = \btau^{\mt}\}$ of symmetric matrices, where $\btau^{\mt} := (\tau_{ji})$ is the transpose of $\btau = (\tau_{ij})$. The component-wise inner product of two matrices $\bsig = (\sigma_{ij})$ and $\btau = (\tau_{ij}) \in \mathbb{R}^{m \times n}$ is given by $\bsig : \btau \coloneq \sum_{i,j} \sigma_{ij} \tau_{ij}$.

Let $D$ be a polyhedral Lipschitz bounded domain of $\bbR^d$ $(d=2,3)$, with boundary $\partial D$. Throughout this work, we apply all differential operators row-wise.  For example, given a tensorial function $\bsig:D\to \mathbb{R}^{d\times d}$ and a vector field $\bu:D\to \bbR^d$, we set the divergence $\bdiv \bsig:D \to \bbR^d$, the gradient $\bnabla \bu:D \to \mathbb{R}^{d\times d}$, and the linearized strain tensor $\beps(\bu) : D \to \mathbb{R}^{d\times d}_{\text{sym}}$ as
\[
	(\bdiv \bsig)_i := \sum_j   \partial_j \sigma_{ij} \,, \quad (\bnabla \bu)_{ij} := \partial_j u_i\,,
	\quad\hbox{and}\quad \beps(\bu) := \frac{1}{2}\left[\bnabla\bu+(\bnabla\bu)^{\mt}\right].
\]
For $s\in \bbR$, $H^s(D, \mathbb{R}^{m\times n})$ represents the usual Hilbertian Sobolev space of functions with domain $D$ and values in $\mathbb{R}^{m\times n}$. In the case $m=n=1$, we simply write $H^s(D)$. The norm of $H^s(D, \mathbb{R}^{m\times n})$ is denoted by $\norm{\cdot}_{s,D}$ and the corresponding semi-norm is written as $|\cdot|_{s,D}$. We adopt the convention  $H^0(D, \mathbb{R}^{m\times n}):=L^2(D,\mathbb{R}^{m\times n})$ and let $(\cdot, \cdot)_D$ be the inner product in $L^2(D, \mathbb{R}^{m\times n})$, i.e.,
\begin{equation*}
	  \inner{\bsig, \btau}_D \coloneq  \int_D \bsig:\btau,\quad \forall \bsig, \btau\in L^2(D,\mathbb{R}^{m\times n}).
\end{equation*} 

The space of vector fields in $L^2(D, \mathbb{R}^d)$ with divergence in $L^2(D)$ is denoted by $H(\div, D)$. The corresponding norm is given by $\norm{\bv}^2_{\div, D}:=\norm{\bv}_{0,D}^2+\norm{\div\bv}^2_{0,D}$. Let $\bn$ be the outward unit normal vector to $\partial D$, the Green formula
\begin{equation}\label{green1}
	\inner{\bv, \nabla q}_D + \inner{\div \bv, q}_D = \int_{\partial D} \bv\cdot\bn q\qquad  \forall q \in H^1(D),
\end{equation}
allows for the extension of the normal trace operator $\bv \to (\bv|_{\partial D})\cdot\bn$ to a linear continuous mapping $(\cdot|_{\partial D})\cdot\bn:\, H(\bdiv, D) \to H^{-\sfrac{1}{2}}(\partial D)$, where $H^{-\sfrac{1}{2}}(\partial D)$ is the dual of $H^{\sfrac{1}{2}}(\partial D)$.

Similarly, we denote by $H(\bdiv, D, \mathbb{R}^{d\times d}_{\text{sym}})$ the space of tensors in $L^2(D, \mathbb{R}^{d\times d}_{\text{sym}})$ with divergence in $L^2(D, \bbR^d)$. Its norm is defined as $\norm{\btau}^2_{\bdiv,D}:=\norm{\btau}_{0,D}^2+\norm{\bdiv\btau}^2_{0,D}$. Again, the Green formula
\begin{equation}\label{green2}
	\inner{\btau, \beps(\bv)}_D + \inner{\bdiv \btau, \bv}_D = \int_{\partial D} \btau\bn\cdot \bv\qquad  \forall \bv \in H^1(D,\bbR^d),
\end{equation}
enables the extension of the normal trace operator $\btau \to (\btau|_{\partial D})\bn$ to a linear continuous mapping $(\cdot|_{\partial D})\bn:\, H(\bdiv, D, \mathbb{R}^{d\times d}_{\text{sym}}) \to H^{-\sfrac{1}{2}}(\partial D, \bbR^d)$.

Since we are dealing with an evolution problem, in addition to the Sobolev spaces defined above, we need to introduce function spaces defined over a bounded time interval $(0, T)$ and taking values in a separable Banach space $V$, whose norm is denoted $\norm{\cdot}_{V}$.  For $1 \leq p \leq\infty$, $L^p_{[0,T]}(V)$ represents the space of Bochner-measurable functions $f: (0,T) \to V$ with finite norms, given by
\[
	\norm{f}^p_{L^p_{[0,T]}(V)}:= \int_0^T\norm{f(t)}_{V}^p\, \text{d}t \quad\hbox{for $1\leq p < \infty$ and} \quad \norm{f}_{L_{[0,T]}^\infty(V)}:= \esssup_{[0, T]} \norm{f(t)}_V.
\]

We denote the Banach space of continuous functions $f: [0,T] \to V$ by $\mathcal{C}^0_{[0,T]}(V)$. Moreover, for any $k \in \mathbb{N}$, $\mathcal{C}^k_{[0,T]}(V)$ represents the subspace of $\mathcal{C}^0_{[0,T]}(V)$ consisting of functions $f$ with strong derivatives $\frac{d^j f}{dt^j}$ in $\mathcal{C}^0_{[0,T]}(V)$ for all $1 \leq j \leq k$. In what follows, we will interchangeably use the notations $\dot{f}:= \frac{d f}{dt}$, $\ddot{f}:= \frac{d^2 f}{dt^2} $, and $\dddot{f}:= \frac{d^3 f}{dt^3} $ to denote the first, second, and third derivatives with respect to $t$. 

Throughout the rest of this paper, we shall use the letter $C$  to denote generic positive constants independent of the mesh size $h$ and the polynomial degree $k$ and the time step $\Delta t$. These constants may represent different values at different occurrences. Moreover, given any positive expressions $X$ and $Y$ depending on $h$ and $k$, the notation $X \,\lesssim\, Y$  means that $X \,\le\, C\, Y$.

\section{Four field formulation of a dynamic and linear poroelastic problem}\label{sec:model}

This section presents the dynamic Biot model for saturated porous media within a polygonal or polyhedral domain \(\Omega \subset \mathbb{R}^d\), where \(d = 2, 3\). The model describes the behavior of a porous material fully saturated with fluid, accounting for the interactions between the solid skeleton and the fluid phase. It is governed by a system of coupled equations that represent the conservation of linear momentum, fluid transport, and mass balance.

The linear momentum balance equation is given by
\begin{equation}\label{momentum}
	\rho_{11} \ddot{\bd}_s + \rho_{12} \dot{\bu}_f -\bdiv(\bsig - \alpha p \mathrm{I}_d) = \bF_s	\quad \text{in $\Omega\times (0, T]$},
\end{equation}
where $\bd_s:\Omega\times [0,T] \to \bbR^d$ represents the skeleton displacement, $\bu_f:\Omega\times [0,T] \to \bbR^d$ denotes the fluid velocity, $p: \Omega \times [0, T] \to \mathbb{R}$ is the pore pressure, and $\bF_s:\Omega\times [0,T] \to \bbR^d$ represents the body force. The coefficient $0< \alpha \leq 1$ is the Biot-Willis parameter. The linearized strain tensor $\beps(\bd_s)$ is linked to the effective stress $\bsig:\Omega\times [0,T] \to \mathbb{R}^{d\times d}_{\text{sym}}$ through Hooke's law:
\begin{equation}\label{Constitutive:solid}
	\bsig  = \cC \beps(\bd_s)  \quad \text{in $\Omega\times (0, T]$},
\end{equation}
where $\cC$ is a symmetric and positive definite fourth-order stiffness tensor.  

Fluid transport follows Darcy's law, simplified here for linearity and isotropy as
\begin{equation}\label{Darcy}
	\rho_{12} \ddot{\bd}_s + \rho_{22} \dot{\bu}_f + \tfrac{\eta}{\kappa} \bu_f + \nabla p = \bF_f	\quad \text{in $\Omega\times (0, T]$},
\end{equation}
where $\bF_f:\Omega\times [0,T] \to \bbR^d$ is a source term and the parameters $\kappa>0$ and $\eta>0$ are the absolute permeability and the dynamic viscosity, respectively. In \eqref{momentum}-\eqref{Darcy}, the coefficients $\rho_{ij}> 0$, $1 \leq i, j \leq 2$, are related to the solid density $\rho_s$, the saturating fluid density $\rho_f$, the porosity $0 < \phi <1$, and the tortuosity $\nu>1$ as follows:
\begin{align}
    \rho_{11} := \phi\rho_f + (1 - \phi)\rho_s, \quad \rho_{12} := \rho_f,\quad \rho_{22} := \frac{\nu}{\phi}\rho_f.
\end{align}
 Finally, mass balance is stated as:
\begin{equation}\label{mass}
	s \dot p + \div \bu_f +\alpha \div \dot\bd_s = g	\quad \text{in $\Omega\times (0, T]$},
\end{equation}
where $s>0$ is the constrained specific storage coefficient and $g:\Omega \to \bbR$ is the source/sink density function of the fluid.  

To fully define the model problem, we complement equations \eqref{momentum}--\eqref{mass} with the initial conditions 
\begin{equation}\label{IC}
	\bd_s(0) = \bd_s^0,\quad \dot\bd_s(0) = \bu_s^0 \quad \bu_f(0) = \bu^0_f,\quad \text{and}\quad p(0) = p^0	\quad \text{in $\Omega$},
\end{equation}
where $\bd^0, \bu_s^0, \bu^0_f :\Omega \to \bbR^d$, and $p^0:\Omega \to \bbR$ are the given initial data for displacement, solid velocity, fluid velocity, and pressure, respectively.

According to our assumption on $\cC$, there exist constants $c^+ > c^- > 0$ such that  
\begin{equation}\label{CD}
c^- \boldsymbol{\zeta}:\boldsymbol{\zeta}\, \leq \cC \boldsymbol{\zeta} :\boldsymbol{\zeta} \leq c^+ \, \boldsymbol{\zeta}:\boldsymbol{\zeta}
\quad \forall \boldsymbol{\zeta}\in \mathbb{R}^{d\times d}_{\text{sym}}.
\end{equation}
Moreover, given the density interaction matrix $R:= \begin{pmatrix}
	\rho_{11} & \rho_{12} \\
	\rho_{12} & \rho_{22}
\end{pmatrix}$, we have that 
\begin{equation}\label{eq:eigenR}
	\rho^- \boldsymbol{\eta}:\boldsymbol{\eta}\, \leq \boldsymbol{\eta}R :\boldsymbol{\eta} \leq \rho^+ \, \boldsymbol{\eta}:\boldsymbol{\eta}
\quad \forall \boldsymbol{\eta}\in \mathbb{R}^{d\times 2},
\end{equation}
where $\rho^+> \rho^- > 0$ are the eigenvalues of $R$.

We formulate the problem in terms of the stress tensor $\bsig$, the pressure $p$, and the velocities $\bu_s:= \dot \bd_s$ and $\bu_f$ of the solid and fluid phases, respectively. Using the notation  $\beta := \eta/\kappa$, equations \eqref{momentum}-\eqref{mass} can be expressed as the following first order system in time:
\begin{subequations}
	\begin{align}
			\rho_{11} \dot{\bu}_s + \rho_{12} \dot{\bu}_f -\bdiv(\bsig - \alpha p \mathrm{I}_d) &= \bF_s	\quad \text{in $\Omega\times (0, T]$},\label{4field-a}
			\\
			\rho_{12} \dot{\bu}_s + \rho_{22} \dot{\bu}_f + \beta \bu_f + \nabla p &= \bF_f	\quad \text{in $\Omega\times (0, T]$},\label{4field-b}
			\\
			\dot\bsig &= \cC \beps(\bu_s)	\quad \text{in $\Omega\times (0, T]$},\label{4field-c}
			\\
			s \dot p + \div \bu_f +\alpha \div \bu_s &= g	\quad \text{in $\Omega\times (0, T]$}.	\label{4field-d}
	\end{align}
\end{subequations}

To impose boundary conditions, we start by defining two disjoint partitions of the boundary $\Gamma:=\partial \Omega$. Specifically, we assume $\Gamma = \Gamma^s_D \cup \Gamma^s_N$ and $\Gamma = \Gamma^f_D \cup \Gamma^f_N$, with both $\Gamma^s_D$ and $\Gamma^f_D$ having positive Lebesgue measures. Let $\bn$ denote the normal unit outward vector on $\Gamma$. We consider the following mixed homogeneous Dirichlet/Neumann boundary conditions:
\begin{equation}\label{BC}
	\begin{aligned}
		\bu_s & = \mathbf 0 \quad \text{on $\Gamma^s_D\times (0, T]$} 
		&\quad (\bsig - \alpha p \mathrm{I}_d)\bn &= \mathbf 0 \quad \text{on $\Gamma^s_N\times (0, T]$},
		\\
        p &= 0 \quad \text{on $\Gamma^f_D\times (0, T]$} &\quad 
		\bu_f\cdot \bn & =  0 \quad \text{on $\Gamma^f_N\times (0, T]$}.
	\end{aligned}
\end{equation}

We derive from \eqref{IC} the following initial conditions for the velocities/stress-pressure formulation \eqref{4field-a}--\eqref{4field-d} of the poroelasticity problem: 
\begin{equation}\label{IC1}
	\bu_s(0) = \bu_s^0,\quad \bu_f(0) = \bu^0_f,\quad \bsig(0) = \bsig^0 \quad \text{and}\quad p(0) = p^0	\quad \text{in $\Omega$},
\end{equation}
where $\bsig^0:= \cC \beps(\bd^0)$.

\section{Well-posedness of the model problem}\label{sec:wellposedness}

In this section, we establish the well-posedness of the model problem \eqref{4field-a}--\eqref{4field-d} with boundary conditions \eqref{BC} and initial conditions \eqref{IC1}, using the theory of strongly continuous semigroups \cite{pazy}.

\subsection{Functional setting}

To facilitate a unified treatment of velocity components, we introduce the notation $\underline{\bu}$ to represent the concatenation of solid and fluid velocities into a single block matrix $\underline{\bu} \coloneq \big[\bu_s \mid \bu_f\big]: \Omega \to \mathbb{R}^{d \times 2}$. This structure allows us to rewrite equations \eqref{4field-a}-\eqref{4field-b} more compactly as
\begin{equation}\label{compact1}
    \big[\dot{\bu}_s \mid \dot{\bu}_f\big]  = \big[\bdiv(\bsig - \alpha p \mathrm{I}_d) \mid - \nabla p - \beta \bu_f \big]R^{-1} + \big[ \bF_s \mid \bF_f \big]R^{-1},
\end{equation}  
where the notation $[\boldsymbol{a} \mid \boldsymbol{b}] \in \mathbb{R}^{d \times 2}$ is generally adopted to represent a $d \times 2$ matrix with columns $\boldsymbol{a}$ and $\boldsymbol{b} \in \mathbb{R}^d$.

In addition, we use the pair $(\bsig, p)$ to collectively denote the stress and pressure variables. It follows from \eqref{4field-c}--\eqref{4field-d} that $(\bsig, p)$ satisfies the evolution equation
\begin{equation}\label{compact2}
	\big(\dot{\bsig}, \dot{p} \big)  + \big( -\mathcal{C} \beps(\bu_s) , \tfrac{1}{s} \left( \div \bu_f +\alpha \div \bu_s \right) \big) = \big(\mathbf 0 , \tfrac{1}{s}g  \big)
	\quad \text{in $\Omega \times (0,T]$ }.
\end{equation}

To establish a suitable framework for formulating the first-order system in time \eqref{compact1}–\eqref{compact2}, we begin by defining the pivot Hilbert spaces $\mathcal{H}_1\coloneq L^2(\Omega, \mathbb R^{d\times 2})$ and $\mathcal{H}_2 \coloneq  L^2(\Omega,\mathbb{R}^{d\times d}_{\text{sym}}) \times   L^2(\Omega)$. We endow $\mathcal{H}_1$ with the inner product $\inner{ \underline{\bu}, \underline{\bv} }_{\mathcal{H}_1} \coloneq  \inner{\underline{\bu} R, \underline{\bv}}_{\Omega}$. According to \eqref{eq:eigenR}, the corresponding norm $\norm{\underline{\bv}}^2_{\mathcal{H}_1} \coloneq  \inner{ \underline{\bv}, \underline{\bv} }_{\mathcal{H}_1}$  satisfies 
\begin{equation}\label{bound:rho}
   \rho^-  (\norm{\bv_s}^2_{0,\Omega} + \norm{\bv_f}_{0,\Omega}^2) \leq 	\norm{\underline{\bv}}^2_{\mathcal{H}_1} \leq \rho^+  (\norm{\bv_s}^2_{0,\Omega} + \norm{\bv_f}_{0,\Omega}^2) \quad \forall \underline{\bv}  =  \big[\bv_s \mid \bv_f\big] \in \mathcal{H}_1.
\end{equation}








Additionally, we define $\mathcal{A} \coloneq  \mathcal{C}^{-1}$, and equip 
$\mathcal{H}_2$ with the inner product  
\[
\inner{ (\bsig,p), (\btau,q) }_{\mathcal{H}_2} \coloneq   \inner{\mathcal{A}\bsig, \btau}_\Omega  + \inner{s p,q }_\Omega,\quad (\bsig,p), (\btau,q) \in \mathcal{H}_2,
\]  
with the associated norm $\norm{(\btau,q) }^2_{\mathcal{H}_2}  \coloneq   \inner{\mathcal{A}\btau, \btau}_\Omega  + \inner{s q,q }_\Omega$. By \eqref{CD}, there exist  constants $a^+ > a^- >0$ such that
\begin{equation}\label{normH}
	a^- \left( \norm{\btau}^2_{0, \Omega}  + \norm{q}^2_{0, \Omega}  \right) \leq \norm{(\btau,q)  }_{\cH_2}^2 \leq a^+ \left( \norm{\btau}^2_{0, \Omega}  + \norm{q}^2_{0, \Omega}  \right)   \quad \forall (\btau,q)    \in \mathcal{H}_2.
\end{equation} 

To incorporate the essential boundary conditions specified in \eqref{BC} within the energy space associated with the fluid velocity component, we introduce the closed subspace
\[
H_N(\div, \Omega) := \set{ \bv\in H(\bdiv, \Omega); \quad \dual{\bv\cdot \bn,q}_{\Gamma}= 0 	\quad \forall q\in H^1_D(\Omega) },
\]
that consists of vector fields in $H(\bdiv, \Omega)$ with a free normal component on $\Gamma^f_N$. Here,  $H^1_D(\Omega):=\set{q\in H^1(\Omega);\ q|_{\Gamma^f_D} = 0}$, and $\dual{\cdot, \cdot}_\Gamma$ represents the duality pairing between $H^{\sfrac12}(\Gamma)$ and $H^{-\sfrac12}(\Gamma)$. 

Similarly, we let $H^1_D(\Omega,\bbR^d):=\set{\bv\in H^1(\Omega,\bbR^d);\ \bv|_{\Gamma^s_D} = \mathbf 0}$ and define
\[
H_N(\bdiv, \Omega, \mathbb{R}^{d\times d}_{\text{sym}}) := \set{ \btau\in H(\bdiv, \Omega, \mathbb{R}^{d\times d}_{\text{sym}}); \quad 
	\left\langle\btau\bn,\bu\right\rangle_{\Gamma}= 0 
	\quad \forall\bu\in H^1_D(\Omega,\bbR^d) },
\]
as the closed subspace of $H(\bdiv, \Omega, \mathbb{R}^{d\times d}_{\text{sym}})$  satisfying a stress--free boundary condition on $\Gamma^s_N$.

The infinitesimal generator $A$ of the strongly continuous semigroup on $\cH_1 \times \cH_2$ associated with the evolution system \eqref{compact1}--\eqref{compact2} is given by
\[
A\left( \underline{\bu} , (\bsig,p)  \right) \coloneq 
\Big( A_1(\bsig,p) + [ \mathbf 0 \mid \beta \bu_f ]R^{-1}, A_2 \underline{\bu} \Big)
\] 
where
\[
A_1(\bsig,p) \coloneq  \big[ -\bdiv(\bsig - \alpha p \mathrm{I}_d) \mid \nabla p  \big] R^{-1}
\quad 
\text{and} 
\quad
A_2\underline{\bu} \coloneq \big(- \mathcal{C} \beps(\bu_s) , \tfrac{1}{s} ( \div \bu_f +\alpha \div \bu_s )  \big).
\]
The domain of the operator $A$ is the subspace  $\mathcal{X}_1 \times \mathcal{X}_2 \subset \mathcal{H}_1 \times \mathcal{H}_2$, where
\[
	\mathcal{X}_1 \coloneq \set{ \underline{\bv}=\big[\bv_s \mid \bv_f\big] ;\  \bv_s \in H_D^1(\Omega,\mathbb{R}^d),\ \bv_f \in H_N(\div,\Omega) }   
\]
and 
\[
	\mathcal{X}_2 \coloneq \set{ (\btau,q) \in L^2(\Omega,\mathbb{R}^{d\times d}_{\text{sym}})\times  H_D^1(\Omega);\quad     \btau - \alpha q \mathrm{I}_d  \in H_N(\bdiv, \Omega, \mathbb{R}^{d\times d}_{\text{sym}})  }
\]
are endowed with the inner products
\[
\inner{\underline{\bu} , \underline{\bv}  }_{\mathcal{X}_1} \coloneq   \inner{\underline{\bu} , \underline{\bv}  }_{\mathcal{H}_1} + \inner{\beps(\bu_s), \beps(\bv_s)}_\Omega + \inner{\tfrac{1}{s}(\div\bu_f + \alpha\div \bu_s), \div\bv_f + \alpha\div \bv_s}_\Omega
\]
and 
\[
	\inner{ (\bsig,p), (\btau,q) }_{\mathcal{X}_2} \coloneq \inner{ (\bsig,p), (\btau,q) }_{\mathcal{H}_2} 
 + \inner*{ \big[ \bdiv(\bsig - \alpha p \mathrm{I}_d ) \mid \nabla q \big]R^{-1} ,   \big[\bdiv(\btau - \alpha q \mathrm{I}_d ) \mid \nabla q \big]  }_\Omega,
\] 
respectively. 

According to our notation, the initial boundary value problem \eqref{4field-a}-\eqref{4field-d} can be written in the following standard form: Find $\underline{\bu}= \big[\bu_s \mid \bu_f\big]: [0,T]\to \mathcal{X}_1$ and $(\bsig,p): [0,T]\to \mathcal{X}_2$ such that  
\begin{align}\label{eq:IVP}
\begin{split}
\frac{\text{d}\underline{\bu} }{\text{d} t}  + A_1(\bsig,p) + \big[ \mathbf 0 \mid \beta \bu_f \big] R^{-1} &= \big[\bF_s \mid \bF_f\big] R^{-1}
\quad \text{in $\Omega \times (0,T]$ } 
\\
\big(\frac{\text{d}\bsig}{\text{d} t}, \frac{\text{d}p}{\text{d} t} \big)  + A_2\underline{\bu} &= \big(\mathbf 0 , \tfrac{1}{s}g  \big)
\quad \text{in $\Omega \times (0,T]$ }
\end{split}
\end{align}
and subject to the  initial conditions
\begin{equation}\label{weakIC}
	\underline{\bu} (0) = \underline{\bu}^0  \quad  \text{and} \quad (\bsig (0), p(0) ) =  (\bsig^0, p^0) \quad \text{in $\Omega$},
\end{equation}
where $\underline{\bu}^0 \coloneq  \big[\bu_s^0 \mid  \bu^0_f \big]$.

\subsection{Application of the Hille-Yosida theory}
Our aim is to prove that the Hille-Yosida theorem can be used to show that the initial boundary value problem \eqref{eq:IVP}--\eqref{weakIC} is well-posed. To achieve this, we first establish that the graph norms defined in $\cX_1$ and $\cX_2$ induce complete topologies.

\begin{proposition}\label{prop:X1Hilbert}
The linear space $\mathcal{X}_1$ endowed with the inner product $\inner{\cdot, \cdot}_{\mathcal{X}_1}$ is a Hilbert space.
\end{proposition}
\begin{proof}
	We only need to prove that $\mathcal{X}_1$ is a Banach space with respect to the norm $\norm{\underline{\bv} }^2_{\mathcal{X}_1} \coloneq \inner{\underline{\bv}, \underline{\bv}  }_{\mathcal{X}_1}$. It is clear that if $\left\{ \underline{\bu}_n \right\} = \left\{[\bu_{s,n} \mid \bu_{f,n}] \right\}$ is a Cauchy sequence with respect to $\norm{\cdot}_{\mathcal{X}_1}$ then
	\begin{subequations}
		\begin{align}
	\bu_{s,n} & \to \bu_s  \quad \text{in $H^1_D(\Omega,\mathbb{R}^d)$ } && (\text{Korn's lemma} ) \label{eq:convH1a}
	\\
	\bu_{f,n} & \to \bu_f \quad \text{in $L^2(\Omega,\mathbb{R}^d)$ } &&  \label{eq:convH1b}
	\\
	\div\bu_{f,n} + \alpha\div \bu_{s,n} & \to w \quad \text{in $L^2(\Omega)$ } &&  \label{eq:convH1c}
		\end{align}
	\end{subequations}
	It follows from \eqref{eq:convH1a} that $\alpha\div \bu_{s,n}  \to \alpha\div \bu_{s} $ in $L^2(\Omega)$. Moreover, \eqref{eq:convH1b} implies that $\div \bu_{f,n}  \to \div \bu_f $ in the space of distributions $D'(\Omega)$ in $\Omega$. By the uniqueness of the limit in $D'(\Omega)$ we deduce that $\div \bu_f = w - \alpha\div \bu_{s} \in L^2(\Omega)$. This proves that $\bu_{f,n}  \to \bu_f$ in $H(\div,\Omega)$ and the continuity of the normal trace operator on $H(\div,\Omega)$ ensures that $\bu_f \in H_N(\div,\Omega)$. We conclude that $\set{\underline{\bu}_n}$ converges to $\underline{\bu} = \big[\bu_s \mid \bu_f\big]$ in $\mathcal{X}_1$, and the result follows.   
\end{proof}

\begin{proposition}\label{prop:X2Hilbert}
	The linear space $\mathcal{X}_2$ endowed with the inner product $\inner{\cdot, \cdot}_{\mathcal{X}_2}$ is a Hilbert space.
\end{proposition}
\begin{proof}
It follows from \eqref{eq:eigenR} that 
	\[
	\frac{1}{\rho^+} \left\{ \norm{\bdiv(\btau - \alpha q \mathrm{I}_d )}^2_{0,\Omega} + \norm{\nabla q}_{0,\Omega}^2   \right\} \leq  \inner*{ [\bdiv(\btau - \alpha q \mathrm{I}_d ) \mid \nabla q] R^{-1} ,  [\bdiv(\btau - \alpha q \mathrm{I}_d ) \mid \nabla q] }_\Omega
	\]
for all $(\btau,q) \in \mathcal{X}_2$. Therefore, if $\set{(\bsig_n, p_n)}$ is a Cauchy sequence in $(\mathcal{X}_2,\norm{\cdot  }_{\mathcal{X}_2} )$, with $\norm{(\btau,q) }^2_{\mathcal{X}_2} \coloneq \inner{(\btau,q), (\btau,q)  }_{\mathcal{X}_2}$, then $\set{p_n}$ converges to $p$ in $H^1_D(\Omega)$, $\set{\bsig_n}$ converges to $\bsig$ in $L^2(\Omega, \mathbb{R}^{d\times d}_{\text{sym}})$, and $\set{\bdiv(\bsig_n - \alpha p_n \mathrm{I}_d)}$ converges to $\mathbf{w}$ in $L^2(\Omega,\mathbb{R}^d)$. Furthermore,  $\set{\bsig_n - \alpha p_n \mathrm{I}_d}$ converges to $\bsig - \alpha p \mathrm{I}_d$ in $L^2(\Omega, \mathbb{R}^{d\times d}_{\text{sym}})$, implying that $\set{\bdiv(\bsig_n - \alpha p_n \mathrm{I}_d)}$ converges to $\bdiv(\bsig - \alpha p \mathrm{I}_d)$ in $D'(\Omega,\mathbb{R}^d)$.  Thus, $\bdiv(\bsig - \alpha p \mathrm{I}_d)= \mathbf{w}  \in  L^2(\Omega,\mathbb{R}^d)$, which ensures the convergence of $\set{\bsig_n - \alpha p_n \mathrm{I}_d}$ to $\bsig - \alpha p \mathrm{I}_d$ in $H_N(\bdiv,\Omega,\mathbb{R}^{d\times d}_{\text{sym}})$. This completes the proof. 
\end{proof}

Establishing the following property of the operator $A: \mathcal{X}_1 \times \mathcal{X}_2 \to  \mathcal{H}_1 \times \mathcal{H}_2$ is essential for the application of the Hille-Yosida theorem: 

\begin{lemma}\label{lem:maximalMonotone}
The linear operator $A: \mathcal{X}_1 \times \mathcal{X}_2 \to  \mathcal{H}_1 \times \mathcal{H}_2$ is maximal monotone. 
\end{lemma}
\begin{proof}
By definition, for all $(\underline{\bu} , (\bsig,p) ) \in \mathcal{X}_1 \times \mathcal{X}_2$ it holds that 
\begin{align}\label{eq:monotone0}
\begin{split}
	\inner{ A\left( \underline{\bu} , (\bsig,p)  \right),  \left( \underline{\bu} , (\bsig,p)  \right) }_{\mathcal{H}_1 \times \mathcal{H}_2} &=
	\inner{\big[ -\bdiv(\bsig - \alpha p \mathrm{I}_d) \mid \nabla p + \beta \bu_f \big] R^{-1}, \underline{\bu}}_{\mathcal{H}_1} 
	\\
	&\qquad + \inner{ (-\mathcal{C} \beps(\bu_s) , \tfrac{1}{s} ( \div \bu_f +\alpha \div \bu_s ) ), (\bsig,p) }_{\mathcal{H}_2} 
	\\
	& = - \inner{\bdiv(\bsig - \alpha p \mathrm{I}_d), \bu_s }_\Omega + \inner{\nabla p + \beta \bu_f, \bu_f}_\Omega   
	\\
	& \quad \quad - \inner{\beps(\bu_s), \bsig}_\Omega + \inner{\div \bu_f +\alpha \div \bu_s, p}_\Omega. 
\end{split}
\end{align}
We notice that $\inner{\alpha \div \bu_s, p} = \inner{\beps(\bu_s), \alpha p \mathrm{I}_d}$ and apply Green's formulas \eqref{green1}-\eqref{green2} to deduce that 
\begin{align}\label{eq:monotone1}
	\begin{split}
		\inner{ A\left( \underline{\bu} , (\bsig,p)  \right),  \left( \underline{\bu} , (\bsig,p)  \right) }_{\mathcal{H}_1 \times \mathcal{H}_2} 
		& = - \inner{\bdiv(\bsig - \alpha p \mathrm{I}_d), \bu_s }_\Omega - \inner{\bsig - \alpha p\mathrm{I}_d, \beps(\bu_s)}_\Omega    
		\\
		&  \, + \inner{\bu_f, \nabla p}_\Omega + \inner{\div \bu_f , p}_\Omega + \inner{\beta \bu_f, \bu_f}_\Omega= \inner{\beta \bu_f,\bu_f}_\Omega \geq 0,
	\end{split}
	\end{align}
which proves that $A: \mathcal{X}_1 \times \mathcal{X}_2 \to  \mathcal{H}_1 \times \mathcal{H}_2$ is monotone. 

It remains to show that the operator $I_{\mathcal{X}_1 \times \mathcal{X}_2} + A: \mathcal{X}_1 \times \mathcal{X}_2 \to  \mathcal{H}_1 \times \mathcal{H}_2$ is surjective, where $I_{\mathcal{X}_1 \times \mathcal{X}_2}$ represents the identity operator in $\mathcal{X}_1 \times \mathcal{X}_2$. Given $\left( \underline{\bu} ,(\bsig,p)  \right)\in  \mathcal{H}_1 \times \mathcal{H}_2$, we should find $\left( \underline{\bu}^* ,(\bsig^*,p^*)  \right)\in  \mathcal{X}_1 \times \mathcal{X}_2$ satisfying 
\[
\left( I_{\mathcal{X}_1 \times \mathcal{X}_2} + A \right) \left( \underline{\bu}^* ,(\bsig^*,p^*)  \right) = \left( \underline{\bu} ,(\bsig,p)  \right).
\]
In other words, $\underline{\bu}^* = \big[\bu_s^* \mid \bu_f^* \big] \in  \mathcal{X}_1$ and $(\bsig^*,p^*) \in   \mathcal{X}_2$ must solve
\begin{subequations} 
\begin{align}
\underline{\bu}^* &=  \big[ \bdiv(\bsig^* - \alpha p^* \mathrm{I}_d) \mid -\nabla p^* - \beta \bu_f^*  \big] R^{-1} +  \underline{\bu}\label{eq:surjective1a}
\\
 (\bsig^*,p^*) &= \big(\mathcal{C}\beps(u_s^*), - \tfrac{1}{s} ( \div \bu_f^* +\alpha \div \bu_s^* )\big)  +  (\bsig,p)  \label{eq:surjective1b}
\end{align}
\end{subequations}
Taking the  $\mathcal{H}_1$-inner product of \eqref{eq:surjective1a} with an arbitrary $\underline{\bv} \in \mathcal{X}_1$, using \eqref{green1}-\eqref{green2}, and rearranging terms we obtain  
\begin{equation}\label{eq:var1}
\inner{\underline{\bu}^* , \underline{\bv} }_{\mathcal{H}_1} + \inner*{ (\bsig^*,p^*) , \big( \beps(\bv_s) , - (\div \bv_f + \alpha \div \bv_s) \big)}_\Omega + \inner{\beta \bu_f^*, \bv_f}_\Omega =  \inner{\underline{\bu} , \underline{\bv} }_{\mathcal{H}_1}.  
\end{equation}
Substituting back \eqref{eq:surjective1b}  in \eqref{eq:var1}  we deduce that $\underline{\bu}^*=[\bu_s^* \mid \bu_f^*] \in \mathcal{X}_1$ solves 
\begin{equation}\label{eq:var2}
	\inner{\underline{\bu}^* , \underline{\bv} }_{\mathcal{X}_1}  + \inner{\beta \bu_f^*, \bv_f}_\Omega =  \inner{\underline{\bu} , \underline{\bv} }_{\mathcal{H}_1} 
	- \inner{\bsig, \beps(\bv_s)}_\Omega 
	+ \inner{ p , \div \bv_f + \alpha \div \bv_s }_\Omega,
\end{equation}
for all $\underline{\bv} =[\bv_s \mid \bv_f]  \in \mathcal{X}_1$. The well-posedness of problem \eqref{eq:var2} is a consequence of Proposition~\ref{prop:X1Hilbert} and the Lax-Milgram lemma. 

We can now use \eqref{eq:surjective1b} to express $(\bsig^*, p^*)$ in terms of $\underline{\bu}^*$. However, the resulting expression does not guarantee that $(\bsig^*, p^*) \in \mathcal{X}_2$. By employing a dual procedure to the first step and taking the $\mathcal{H}_2$-inner product of \eqref{eq:surjective1b} with an arbitrary $(\btau, q) \in \mathcal{X}_2$, we obtain:
\[
\inner{(\bsig^*,p^*), (\btau , q )}_{\mathcal{H}_2} - \inner{\beps(\bu_s^*), \btau - \alpha q \mathrm{I}_d}_\Omega + \inner{\div \bu_f^*, q}_\Omega =   \inner{(\bsig, p), (\btau , q )}_{\mathcal{H}_2}.
\]  
Using \eqref{green1}-\eqref{green2} we deduce that 
\[
	\inner{(\bsig^*,p^*), (\btau , q )}_{\mathcal{H}_2} + \inner*{ \underline{\bu} ^* , \big[ \bdiv(\btau - \alpha q \mathrm{I}_d ) \mid -\nabla q\big] }_\Omega =   \inner{(\bsig, p), (\btau , q )}_{\mathcal{H}_2}.
\]
Combining now \eqref{eq:surjective1a} with the last equation shows that $(\bsig^*, p^*) \in \mathcal{X}_2$ satisfies the variational problem 
\begin{equation}\label{eq:var3}
	\inner{(\bsig^*,p^*), (\btau , q )}_{\mathcal{X}_2} = \inner{(\bsig, p), (\btau , q )}_{\mathcal{H}_2} - \inner{  \big[\mathbf 0 \mid \beta \bu_f^*\big] R^{-1}, \big[ \bdiv(\btau - \alpha q \mathrm{I}_d ) \mid \nabla q\big]}_\Omega \quad  \forall  (\btau , q )  \in \mathcal{X}_2,
\end{equation}
whose well-posedness is a consequence of Proposition~\ref{prop:X2Hilbert} and the Lax-Milgram lemma. 

We have then shown that solving successively problems \eqref{eq:var2} and \eqref{eq:var3} provides an inverse image $\left( \underline{\bu}^* ,(\bsig^*,p^*)  \right)\in  \mathcal{X}_1 \times \mathcal{X}_2$ of $\left( \underline{\bu},(\bsig, p)  \right)\in  \mathcal{H}_1 \times \mathcal{H}_2$ under $I_{\mathcal{X}_1 \times \mathcal{X}_2} + A$, which proves the result. 
\end{proof}

\begin{remark}\label{rem:density}
It is worth noting that Lemma~\ref{lem:maximalMonotone} ensures that $\mathcal{X}_1 \times \mathcal{X}_2$ is a dense subspace of $\mathcal{H}_1 \times \mathcal{H}_2$, see \cite[Lemma 76.4]{ErnBook2021III}.
\end{remark}

We are now in a position to provide the main result of this section. 
\begin{theorem}\label{thm:Hille-Yosida}
For all $\underline{\bF} := [\bF_s \mid \bF_f]  \in \mathcal{C}^1_{[0,T]}(L^2(\Omega, \mathbb{R}^{d\times 2}))$, $g \in \mathcal{C}^1_{[0,T]}(L^2(\Omega))$, $\underline{\bu}^0 \in \mathcal{X}_1$, and $(\bsig^0,p^0) \in \mathcal{X}_2$, there exist unique $\underline{\bu} \in \mathcal{C}^1_{[0,T]}(\mathcal{H}_1) \cap \mathcal{C}^0_{[0,T]}(\mathcal{X}_1)$ and $(\bsig,p) \in \mathcal{C}^1_{[0,T]}(\mathcal{H}_2) \cap \mathcal{C}^0_{[0,T]}(\mathcal{X}_2)$ solutions to the initial boundary value problem \eqref{eq:IVP}--\eqref{weakIC}. Moreover, there exists a constant $C>0$ such that  
\begin{equation}\label{eq:stab}
\max_{[0,T]}\norm{\underline{\bu}(t)}_{\mathcal{H}_1} + \max_{[0,T]}\norm{(\bsig, p)(t) }_{\mathcal{H}_2} \leq C \left( \max_{[0,T]}\norm{\underline{\bF}(t)}_{0,\Omega} + \max_{[0,T]}\norm{g(t)}_{0,\Omega}   \right) +   \norm{\underline{\bu}^0}_{\mathcal{H}_1} + \norm{(\bsig^0, p^0)} _{\mathcal{H}_2}.
\end{equation}
\end{theorem}
\begin{proof}
See \cite[Theorem 76.7]{ErnBook2021III} for a detailed proof.
\end{proof}

In the following sections, we will develop a finite element discretization method based on the following weak form of \eqref{eq:IVP}: Find $\underline{\bu} = [\bu_s|\bu_f] \in \mathcal{C}_{[0,T]}^1(\cH_1) \cap \mathcal{C}^0_{[0,T]}(\cX_1)$ and  $(\bsig, p) \in \mathcal{C}^1_{[0,T]}(\cH_2) \cap \mathcal{C}^0_{[0,T]}(\cX_2)$ satisfying 
\begin{align}\label{eq:weakform}
	\begin{split}
		\inner{\dot{\underline{\bu}}, \underline{\bv}}_{\cH_1} + \inner{\beta\bu_f,\bv_f}_\Omega  -\inner{\bdiv(\bsig - \alpha p \mathrm{I}_d), \bv_s}_\Omega  
		+ \inner{\nabla p, \bv_f}_\Omega
		&= \inner{\bF_s, \bv_s}_\Omega + \inner{\bF_f, \bv_f}_\Omega
		\\
	\inner{(\dot\bsig, \dot p), (\btau,q) }_{\cH_2} 
	+  \inner{ \bdiv(\btau - \alpha q \mathrm{I}_d) , \bu_s}_\Omega  
	-  \inner{ \nabla q ,\bu_f}_\Omega
	&= \inner{g, q}_\Omega,
	\end{split}
\end{align}
for all $\underline{\bv} = [\bv_s \mid \bv_f]\in \cX_1$ and all $(\btau, q)\in \cX_2$, and subject to the initial conditions \eqref{weakIC}.

\section{Finite element spaces and approximation properties}\label{sec:FE}

For the sake of simplicity, from now on we assume that $\Gamma^s_D = \Gamma^f_D = \Gamma$. As a result, the boundary conditions \eqref{BC} become 
\begin{equation}\label{newBC}
		\bu_s = \mathbf 0  \quad \text{and} 
	\quad 
	p = 0 \quad \text{on $ \Gamma\times (0, T]$}.
\end{equation}   

This section revisits well-established  $hp$--approximation properties, adapting them to match our specific notations and forthcoming requirements.

Let $\cT_h$ be a shape regular partition of the domain $\bar \Omega$ into tetrahedra and/or parallelepipeds if $d=3$ and into triangles and/or quadrilaterals if $d=2$. We allow $\cT_h$ to have hanging nodes. We denote by $h_K$ the diameter of $K$ and let the parameter $h:= \max_{K\in \cT_h} \{h_K\}$ represent the size of the mesh $\cT_h$.

We define a closed subset $F\subset \overline{\Omega}$ as an interior edge/face if it has a positive $(d-1)$-dimensional measure and can be expressed as the intersection of the closures of two distinct elements $K$ and $K'$, ie, $F =\bar K\cap \bar K'$. On the other hand, a closed subset $F\subset \overline{\Omega}$ is a boundary edge/face if there exists $K\in \cT_h$ such that $F$ is an edge/face of $K$ and $F =  \bar K\cap \partial \Omega$. We consider the set $\cF_h^0$ of the interior edges/faces and the set $\cF_h^\partial$ of the boundary edges/faces and let $\cF_h = \cF_h^0\cup \cF_h^\partial$. We denote by $h_F$ the diameter of an edge/face $F\in\cF_h$ and assume that $\cT_h$ is locally quasi-uniform with constant $\gamma>0$. This means that, for all $h$ and all $K\in \cT_h$, we have that
\begin{equation}\label{reguT}
	h_F \leq h_K\leq \gamma h_F\quad \forall F\in \cF(K),
\end{equation}
where $\cF(K)$ represents the set  of edges/faces composing the element $K\in \cT_h$. This condition implies that the neighboring elements have similar sizes.

For all $s\geq 0$, the broken Sobolev space with respect to the partition $\cT_h$ of $\bar \Omega$ is defined as
\[
	H^s(\cT_h,\mathbb{R}^{m\times n}):=
	\set{\bv \in L^2(\Omega, \mathbb{R}^{m\times n});\ \bv|_K\in H^s(K, \mathbb{R}^{m\times n})\quad \forall K\in \cT_h}.
\]
Following the convention mentioned above, we write $H^s(\cT_h,\bbR) = H^s(\cT_h)$ and $H^0(\cT_h,\mathbb{R}^{m\times n}) = L^2(\cT_h,\mathbb{R}^{m\times n})$. We introduce the inner product
\[
	\inner{\boldsymbol{\psi}, \boldsymbol{\varphi}}_{\cT_h} := \sum_{K\in \cT_h} \inner{\boldsymbol{\psi}, \boldsymbol{\varphi}}_{ K}\quad \forall  \boldsymbol{\psi}, \boldsymbol{\varphi} \in L^2(\cT_h, \mathbb{R}^{m\times n})
\]
and write $\norm{\boldsymbol{\psi}}^2_{0,\cT_h}:= \inner{\boldsymbol{\psi}, \boldsymbol{\psi}}_{\cT_h}$. Accordingly, we let $\partial \cT_h:=\set{\partial K;\ K\in \cT_h}$ be the set of all element boundaries and define $L^2(\partial \cT_h,\mathbb{R}^{m\times n})$ as the space of $m\times n$ matrix-valued functions that are square-integrable in each $\partial K\in \partial \cT_h$. We define
\[
	\dual{\boldsymbol{\psi}, \boldsymbol{\varphi}}_{\partial \cT_h} := \sum_{K\in \cT_h} \dual{\boldsymbol{\psi},\boldsymbol{\varphi}}_{\partial K},
	\quad \text{and} \quad
	\norm{\boldsymbol{\varphi}}^2_{0, \partial \cT_h}:= \dual{\boldsymbol{\varphi}, \boldsymbol{\varphi}}_{\partial \cT_h}
	\quad
	\forall  \boldsymbol{\psi}, \boldsymbol{\varphi}\in L^2(\partial \cT_h,\mathbb{R}^{m\times n}),
\]
where  $\dual{\boldsymbol{\psi}, \boldsymbol{\varphi}}_{\partial K} := \sum_{F\in \cF(K)} \int_F \boldsymbol{\psi}: \boldsymbol{\varphi}$.  Besides, we equip the space $L^2(\cF_h,\mathbb{R}^{m\times n})$ with the inner product
\[
	(\boldsymbol{\psi}, \boldsymbol{\varphi})_{\cF_h} := \sum_{F\in \cF_h} \int_F\boldsymbol{\psi}: \boldsymbol{\varphi} \quad \forall \boldsymbol{\psi}, \boldsymbol{\varphi}\in L^2(\cF_h,\mathbb{R}^{m\times n}),
\]
and denote the corresponding norm $\norm*{\boldsymbol{\varphi}}^2_{0,\cF_h}:= (\boldsymbol{\varphi},\boldsymbol{\varphi})_{\cF_h}$. 

Hereafter,  $\cP_\ell(D)$ is the space of polynomials of degree at most $\ell\geq 0$ on $D$ if $D$ is a triangle/tetrahedron, and the space of polynomials of degree at most $\ell$ in each variable  if $D$ is a quadrilateral/parallelepiped.  The space of $\bbR^{m\times n}$-valued functions with components in $\cP_\ell(D)$ is denoted $\cP_\ell(D,\bbR^{m\times n})$. In particular, $\cP_\ell(D,\bbR^{d\times d}_{\text{sym}})$ refers to symmetric $d\times d$-matrices with components in $\cP_\ell(D)$. We introduce the space of piecewise-polynomial functions
\[
	\cP_\ell(\cT_h) :=
	\set{ v\in L^2(\cT_h): \ v|_K \in \cP_\ell(K),\ \forall K\in \cT_h }
\]
with respect to the partition $\cT_h$ and the space of piecewise-polynomial functions
\[
	\cP_\ell(\cF_h) :=
	\set{ \hat v\in L^2(\cF_h): \ \hat v|_F \in \cP_\ell(F),\ \forall F\in \cF_h }
\]
with respect to the skeleton $\cF_h$ of the mesh $\cT_h$. The subspace of $L^2(\cT_h, \bbR^{m\times n})$ with components in $\cP_\ell(\cT_h)$ is denoted $\cP_\ell(\cT_h, \bbR^{m\times n})$. Likewise, $\cP_\ell(\cF_h, \bbR^{m\times n})$ stands for the subspace of $L^2(\cF_h, \bbR^{m\times n})$ with components in $ \cP_\ell(\cF_h)$. We finally consider 
\[
\cP_\ell(\partial \cT_h, \bbR^{m\times n}) := \set*{\phi\in L^2(\partial \cT_h, E);\ \phi|_{\partial K}\in  \cP_\ell(\partial K,\bbR^{m\times n}),\ \forall K\in \cT_h},
\]
where  $\cP_\ell(\partial K, \bbR^{m\times n}):= \prod_{F\in \cF(K)} \cP_\ell(F, \bbR^{m\times n})$. 

\begin{remark}
It is important to keep in mind that, by definition,  the functions in $L^2(\partial \cT_h,\bbR^{m\times n})$ and $\cP_\ell(\partial \cT_h, \bbR^{m\times n})$, are multi-valued on every interior face $F$, whereas the functions in $L^2(\cF_h,\bbR^{m\times n})$ and $\cP_\ell(\cF_h, \bbR^{m\times n})$ are single-valued on each face $F$.
\end{remark}

We consider $\bn\in \cP_0(\partial \cT_h, \bbR^d)$, where $\bn|_{\partial K}=\bn_K $ is the unit normal vector of $\partial K$ oriented toward the exterior of $K$. Obviously, if $F = K\cap K'$ is an interior edge/face of $\cF_h$, then $\bn_K = -\bn_{K'}$ on $F$. If $\boldsymbol{\varphi} \in H^s(\cT_h, \bbR^{m\times n})$,  with $s>\sfrac{1}{2}$, the function $\boldsymbol{\varphi}|_{\partial \cT_h}\in L^2(\partial \cT_h,\bbR^{m\times n})$ is meaningful by virtue of the trace theorem. For the same reason, if $\boldsymbol{\varphi} \in H^1(\Omega, \bbR^{m\times n})$  then $\hat{\boldsymbol{\varphi}} \coloneq  \boldsymbol{\varphi}|_{\cF_h}$ is well defined in $L^2(\cF_h, \bbR^{m\times n})$.

For $k\geq 0$, we introduce the finite-dimensional subspaces of $\mathcal{H}_1$ and $\mathcal{H}_2$ given by
\[
	\mathcal{H}_{1,h} \coloneq  \cP_{k+1}(\cT_h,\bbR^{d\times 2})\quad \text{and} \quad \mathcal{H}_{2,h} \coloneq \cP_{k}(\cT_h,\mathbb{R}^{d\times d}_{\text{sym}}) \times \cP_{k}(\cT_h),
\]
respectively. 

We consider the following discrete trace inequality.
\begin{lemma}\label{TraceDG}
	There exists a constant $C>0$ independent of $h$ and $k$ such that
\begin{equation}\label{discTrace}
		 \norm{\tfrac{h^{\sfrac{1}{2}}_{\cF}}{k+1} q}_{0,\partial \cT_h} \leq C \norm*{ q}_{0, \cT_h}\quad \forall  q \in \cP_{k}(\cT_h).
\end{equation}
\end{lemma}
\begin{proof} See \cite[Lemma 3.2]{meddahi2023hp}.
\end{proof}

For any integer $\ell\geq 0$ and $K\in \cT_h$, we denote by $\Pi_K^\ell$ the $L^2(K)$-orthogonal projection onto $\cP_{\ell}(K)$. The global projection $\Pi^\ell_\cT$ in $L^2(\cT_h)$ onto $\cP_\ell(\cT_h)$ is then given by $(\Pi^\ell_\cT v)|_K = \Pi_K^\ell(v|_K)$ for all $K\in \cT_h$. Similarly, the global projection $\Pi_\cF^\ell$ in $L^2(\cF_h)$ onto $\cP_{\ell}(\cF_h)$ is given,  separately for all $F\in \cF_h$,  by $(\Pi^\ell_\cF \hat v)|_{F} = \Pi_F^\ell(\hat v|_F)$, where $\Pi^\ell_F$ is the $L^2(F)$ orthogonal projection onto $\cP_\ell(F)$. In the following, we maintain the notation $\Pi_\cT^\ell$ to refer to the $L^2$-orthogonal projection onto $\mathcal{P}_\ell(\mathcal{T}_h, \mathbb{R}^{m\times n})$. It should be noted that the tensorial version of $\Pi_\cT^\ell$ inherently preserves the symmetry of the matrices, as it is derived by applying the scalar operator component-wise. Similarly, we will also use $\Pi_\cF^\ell$ to denote the $L^2$ orthogonal projection onto $\mathcal{P}_\ell(\mathcal{F}_h, \mathbb{R}^{m\times n})$.

In the remainder of this section, we provide approximation properties for the projectors defined above. A detailed proof of these results can be found in \cite[Section 3]{meddahi2023hp} and the references therein.

\begin{lemma}\label{lem:maintool}
	There exists a constant $C>0$ independent of $h$ and $k$  such that
	\begin{equation}\label{tool1}
		\norm{q - \Pi^k_\cT q}_{0,\cT_h} + \norm{ \tfrac{h^{\sfrac{1}{2}}_{\cF}}{k+1}(q - \Pi_\cT^k q)}_{0,\partial \cT_h}
		\leq
		C \tfrac{h_K^{\min\{ r, k \}+1}}{(k+1)^{r+1}}  \norm{q}_{1+r,\Omega},
	\end{equation}
	for all $q \in  H^{1+r}(\Omega)$, with $r\geq 0$.
\end{lemma}
\begin{proof}
	See \cite[Lemma 3.3]{meddahi2023hp}.
\end{proof}

We introduce the spaces 
\[
\mathcal{U} \coloneq \set*{\ubv=[\bv_s \mid \bv_f];\ \bv_s \in H^1(\cT_h,\mathbb{R}^d),\ \bv_f \in H(\div,\cT_h)\cap H^r(\cT_h,\mathbb{R}^d)} \quad \text{with}\ r>\sfrac12,  
\]
\[
\text{and} \quad \hat{\mathcal{U}} \coloneq \set*{\hat{\underline{\bv} } = [\hat \bv_s \mid \hat \bv_f];\ \hat \bv_s \in L^2(\cF^0_h,\mathbb{R}^{d}),\ \hat \bv_f \in L^2(\cF_h,\mathbb{R}^{d}) },
\] 
where
\(
	L^2(\cF^0_h,\mathbb{R}^{d}) \coloneq  \set*{\boldsymbol{\phi}\in L^2(\cF_h,\mathbb{R}^{d});\ \boldsymbol{\phi}|_F = \mathbf 0,\ \forall F\in \cF_h^\partial}
\), 
and endow the product space $\mathcal{U} \times \hat{\mathcal{U}}$ with the semi-norm defined, for all $(\ubv, \hat \ubv)\in \mathcal{U} \times \hat{\mathcal{U}}$, by  
\begin{equation}\label{norm:sym}
	\abs*{(\ubv, \hat \ubv)}^2_{\mathcal{U} \times \hat{\mathcal{U}}} = \norm{\beps( \bv_s)}^2_{0,\cT_h} + \norm{\div \bv_f}^2_{0,\cT_h} + \norm{\tfrac{k+1}{h_\cF^{\sfrac{1}{2}}}(\ubv - \hat \ubv)}^2_{0, \partial \cT_h},
\end{equation}
where $h_\cF\in \cP_0(\cF_h)$ is given by $h_\cF|_F := h_F$ for all $F \in \cF_h$.
For $k\geq 0$, we consider the finite-dimensional subspace  
\[
\hat{\mathcal{H}}_{1,h} \coloneq  \set*{[\hat \bv_s \mid \hat \bv_f];\ \hat \bv_s \in \cP_{k+1}(\cF^0_h, \mathbb{R}^d),\ \hat \bv_f \in \cP_{k+1}(\cF_h, \mathbb{R}^d)} \subset \hat{\mathcal{U}},
\]
where $\cP_{k+1}(\cF^0_h, \mathbb{R}^d) := \set{\boldsymbol{\phi} \in \cP_{k+1}(\cF_h, \mathbb{R}^d);\ \boldsymbol{\phi}|_F = \mathbf 0,\ \forall F\in \cF_h^\partial}$. 

Finally, we  consider the subspace $\mathcal{U}^c$ of $\mathcal{U}$ of functions with continuous traces across the interelements of $\cT_h$, namely, 
\[
\mathcal{U}^c \coloneq  \set*{\ubv = [\bv_s \mid \bv_f];\ \bv_s \in  H_0^1(\Omega,\mathbb{R}^d),\ \bv_f \in H(\bdiv,\Omega)\cap H^r(\Omega,\mathbb{R}^d)}.
\] 

\begin{lemma}\label{maintool2}
	There exists a constant $C>0$ independent of $h$ and $k$  such that
	\begin{equation}\label{tool2}
		\abs*{ (\underline{\bu} - \Pi^{k+1}_{\cT} \underline{\bu}, \underline{\bu}|_{\partial \cF_h} - \Pi^{k+1}_{\cF} (\underline{\bu}|_{\partial \cF_h}) )}_{\mathcal{U} \times \hat{\mathcal{U}}}
		\leq
		C \tfrac{h^{\min\{ r, k \}+1}}{(k+1)^{r+\sfrac{1}{2}}}  \norm{\underline{\bu}}_{2+r,\Omega},
	\end{equation}
	for all  $\underline{\bu}  \in \mathcal{U}^c\cap H^{2+r}(\Omega,\bbR^{d\times 2})$ , $r \geq 0$.
\end{lemma}
\begin{proof}
	The result is proved similarly to \cite[Lemma 3.4]{meddahi2023hp}. 
\end{proof}

\section{The HDG semi-discrete method}\label{sec:semi-discrete}

In this section, we present the HDG semi-discrete method for the poroelasticity problem \eqref{eq:weakform}. We establish the well-posedness of the resulting algebraic-differential equations and prove the method's consistency with the continuous formulation.

We propose the following HDG space-discretization method for problem \eqref{eq:weakform}:  find    $(\underline{\bu}_{h}, \underline{\hat{\bu}}_{h} ) \in \cC^1_{[0,T]}(\mathcal{H}_{1,h} \times \hat{\mathcal{H}}_{1,h})$ and $(\bsig_h, p_h) \in  \cC^1_{[0,T]}(\cH_{2,h})$    satisfying
\begin{align}\label{sd}
	\begin{split}
		\inner*{\dot{\underline{\bu}}_{h}, \underline{\bv}}_{\cH_1} + 
		\inner{(\dot{\bsig}_h, \dot p_h), (\btau,q) }_{\cH_2}  &+ \inner{\beta \bu_{f,h},\bv_f}_\Omega
		+ B_h((\bsig_h, p_h), (\underline{\bv}, \underline{\hat{\bv}} )) 
		- B_h((\btau,q), (\underline{\bu}_{h}, \underline{\hat{\bu}}_{h} )) 
		\\
	  & \quad +\dual{ \tfrac{(k+1)^2}{h_\cF}(\underline{\bu}_{h} - \underline{\hat{\bu}}_{h} ), 
	  \underline{\bv} - \underline{\hat{\bv}} }_{\partial \cT_h}  
	  =  \inner{\underline{\bF}, \underline{\bv}}_\Omega + \inner{g, q}_\Omega,
	\end{split}
\end{align}
for all $(\underline{\bv}, \underline{\hat{\bv}} ) \in \mathcal{H}_{1,h} \times \hat{\mathcal{H}}_{1,h}$ and $(\btau,q) \in  \cH_{2,h}$,  where the bilinear form $B_h(\cdot,\cdot)$ is given by
\begin{align*}
\begin{split}
	B_h( (\btau,q), (\underline{\bv}, \underline{\hat{\bv}} )) &:= 
	\inner*{\btau - \alpha q \mathrm{I}_d, \beps(\bv_s) }_{\cT_h} -  
	\inner*{q, \div\bv_f}_{\cT_h}  
	\\
	& \qquad  - \dual*{(\btau -  \alpha q \mathrm{I}_d)\bn , \bv_{s} - \hat\bv_{s} }_{\partial \mathcal{T}_h} + \dual*{q\bn, \bv_{f} - \hat\bv_{f} }_{\partial \mathcal{T}_h}.
\end{split}
\end{align*}

We start up problem \eqref{sd} with the initial conditions
\begin{equation}\label{initial-R1-R2-h*c}
	\ubu_h(0) = \Pi^{k+1}_\cT \ubu^0, \quad \hat{\ubu}_h(0)  = \Pi^{k+1}_\cF (\ubu^0|_{\partial \mathcal{F}_h}), \quad \bsig_{h}(0)= \Pi_\cT^k \bsig^0, \quad \text{and} \quad p_h(0) = \Pi_\cT^k p^0.
\end{equation}

We have the following boundedness property for the bilinear form $B_h(\cdot,\cdot)$.
\begin{proposition}
	There exists a constant $C>0$ independent of $h$ and $k$ such that
	\begin{equation}\label{Bhh}
		|B_h((\btau_h,q_h)  , (\underline{\bv}, \underline{\hat{\bv}} ))| \leq C  \norm{(\btau_h,q_h) }_{\cH_2}  \abs{ (\underline{\bv}, \underline{\hat{\bv}} )}_{\mathcal{U}\times \hat{\mathcal{U}}}\quad \text{for all}\ (\btau_h,q_h)  \in \mathcal H_{2,h} \ \text{and} \ (\underline{\bv}, \underline{\hat{\bv}} ) \in \mathcal{U}\times \hat{\mathcal{U}}.
	\end{equation}
\end{proposition}
\begin{proof}
	Applying the Cauchy-Schwarz inequality and \eqref{normH}, we deduce that
	\begin{equation}\label{Bh}
		|B_h((\btau,q)  , (\underline{\bv}, \underline{\hat{\bv}} ))| \lesssim ( \norm{(\btau,q) }^2_{\cH_2} + \norm{\tfrac{h_\cF^{\sfrac{1}{2}}}{k+1} \btau}^2_{0,\partial \cT_h} + \norm{\tfrac{h_\cF^{\sfrac{1}{2}}}{k+1} q}^2_{0,\partial \cT_h})^{\sfrac{1}{2}} \abs{ (\underline{\bv}, \underline{\hat{\bv}} )}_{\mathcal{U}\times \hat{\mathcal{U}}},
	\end{equation}
	for all $(\btau,q) \in \cH_2$ such that $\btau \in H^s(\cT_h, \mathbb{R}^{d\times d}_{\text{sym}})$ and $q \in H^s(\cT_h)$ with $s\geq \sfrac{1}{2}$, and for all $(\underline{\bv}, \underline{\hat{\bv}} ) \in \mathcal{U}\times \hat{\mathcal{U}}$. The result follows by applying the discrete trace inequality \eqref{discTrace}.
\end{proof}

\begin{proposition}
	Problem~\eqref{sd}-\eqref{initial-R1-R2-h*c} admits a unique solution.
\end{proposition}
\begin{proof}
	We consider the lineal application $\mmean{\cdot}: \mathcal{H}_{1,h} \to \hat{\mathcal{H}}_{1,h}$ defined by $\mmean{ \ubv } = [ \mmean{\bv_s} | \mmean{\bv_f} ]$, where $\mmean{\bv_s}|_F = \tfrac12 ( \bv_s|_K + \bv_s|_{K'} )|_F$ and $\mmean{\bv_f}|_F = \tfrac12 ( \bv_f|_K + \bv_f|_{K'} )|_F$ for all $F\in \cF_h^0$ with $F = K\cap K'$, and $\mmean{\bv_s}|_F = 0$ and $\mmean{\bv_f}|_F = \bv_f|_F$ for all $F\in \cF_h^\partial$.  
	

	The algebraic differential equation \eqref{sd} can be split  into the following system of equations:
\begin{align}\label{ADE1}
	\begin{split}
	\inner*{\dot{\underline{\bu}}_{h}, \underline{\bv}}_{\cH_1} &+ 
		\inner{\dot{\underline{\bsig}}_h, (\btau,q) }_{\cH_2}  + \inner{\beta \bu_{f,h},\bv_f}_\Omega
		+ B_h((\bsig_h, p_h), (\underline{\bv}, \mathbf 0 )) - B_h((\btau,q), (\underline{\bu}_{h}, \underline{\hat{\bu}}_{h} )) 
		\\
	  & \quad +\dual{ \tfrac{(k+1)^2}{h_\cF}(\underline{\bu}_{h} - \underline{\hat{\bu}}_{h} ), \underline{\bv}  }_{\partial \cT_h}  = \inner{\underline{\bF}, \underline{\bv}}_\Omega + \inner{g, q}_\Omega, \quad \forall \underline{\bv} \in \mathcal{H}_{1,h},\ \forall (\btau,q) \in \mathcal{H}_{2,h}
\\
&B_h((\bsig_h, p_h), (\mathbf 0, \hat{\underline{\bv} }) ) + \dual{\tfrac{(k+1)^2}{h_\cF}(\underline{\bu}_{h} - \underline{\hat{\bu}}_{h} ), - \hat{\underline{\bv} }  }_{\partial \cT_h} = 0 \quad \forall \hat{\underline{\bv} } \in \hat{\mathcal{H}}_{1,h}.
\end{split}
\end{align}  
From the second equation in \eqref{ADE1}, we deduce that the numerical flux $\hat{\underline{\bu}}_h \in \hat{\mathcal{H}}_{1,h}$ satisfies the equation
\[
\dual{ \hat{\underline{\bu}}_h, \hat{\underline{\bv} } }_{\cF_h}  = \dual{ \mmean{\underline{\bu}_h} -  \tfrac{h_\cF}{(k+1)^2} [ \mmean{(\bsig_h - \alpha p_h \mathrm{I}_d)\bn} | - \mmean{p_h\bn} ], \hat{\underline{\bv} } }_{\cF_h}\quad \forall \hat{\underline{\bv} } \in \hat{\mathcal{H}}_{1,h}.
\]
In other words, $\hat{\underline{\bu}}_h$ is the $L^2$-orthogonal projection of $\mmean{\underline{\bu}_h} -  \tfrac{h_\cF}{(k+1)^2} [ \mmean{(\bsig_h - \alpha p_h \mathrm{I}_d)\bn} | - \mmean{p_h\bn} ]$ onto $\hat{\mathcal{H}}_{1,h}$. Substituting this expression for  $\hat{\underline{\bu}}_h$ into the first equation of \eqref{ADE1} yields a system of ordinary differential equations  with unknowns $\ubu_h \in \mathcal{C}^1_{[0,T]}(\mathcal{U}_h)$ and $(\bsig_h, p_h) \in  \mathcal{C}^1_{[0,T]}(\mathcal{H}_{2,h})$,  along with initial conditions: 
\begin{equation}\label{initODE}
	\ubu_h(0) = \Pi^{k+1}_\cT \ubu^0, \quad \bsig_{h}(0)= \Pi_\cT^k \bsig^0, \quad \text{and} \quad p_h(0) =  \Pi_\cT^k p^0.
\end{equation}

The well-posedness of the resulting first order ODE system follows from the fact that $\inner{\cdot, \cdot}_{\mathcal{H}_1}$ and $\inner{\cdot, \cdot}_{\mathcal{H}_2}$ are inner products on the finite--dimensional function spaces $\mathcal{U}_h$ and $\mathcal{H}_{2,h}$, respectively.  
\end{proof}

Let us now verify that the HDG scheme \eqref{sd} is consistent with problem \eqref{eq:weakform}. 

\begin{proposition}\label{consistency}
Let $\underline{\bu} = [\bu_s|\bu_f] \in \mathcal{C}_{[0,T]}^1(\cH_1) \cap \mathcal{C}^0_{[0,T]}(\cX_1)$ and $(\bsig , p) \in \mathcal{C}^1_{[0,T]}(\cH_2) \cap \mathcal{C}^0_{[0,T]}(\cX_2)$ be the solutions of \eqref{eq:weakform}. Assume that $\bsig - \alpha p \mathrm{I}_d\in\cC_{[0,T]}^0(H^s(\cT_h, \bbS))$ and $\underline{\bu}  \in \cC_{[0,T]}^0(H^s(\cT_h, \mathbb{R}^{d \times 2}))$, with $s>\sfrac{1}{2}$. Then, it holds true that 
\begin{align}\label{consistent}
	\begin{split}
		&\inner*{\dot{\underline{\bu}}, \underline{\bv}_h}_{\cH_1} + 
		\inner{(\dot\bsig, \dot p ), (\btau_h,q_h) }_{\cH_2}  + \inner{\beta \bu_{f},\bv_{f,h}}_\Omega
		+ B_h((\bsig,p), (\underline{\bv}_{h}, \underline{\hat{\bv}}_{h} ))  
		\\& \quad 
	  - B_h((\btau_h,q_h), (\underline{\bu}, \underline{\bu}|_{\partial \cF_h} )) 
	  + \dual{ \tfrac{(k+1)^2}{h_\cF}(\underline{\bu} - \underline{\bu}|_{\partial \cF_h}) , \underline{\bv}_{h} - \underline{\hat{\bv}}_{h} }_{\partial \cT_h}  = \inner{\underline{\bF}, \underline{\bv}_h}_\Omega + \inner{g, q_h}_\Omega,
	\end{split}
\end{align}
for all $(\underline{\bv}_{h}, \underline{\hat{\bv}}_{h} ) \in \mathcal{H}_h \times \hat{\mathcal{H}}_h$ and $(\btau_h,q_h) \in  \cH_{2,h}$.
\end{proposition}
\begin{proof}
	The continuity of the normal components of $\bsig -  \alpha p \mathrm{I}_d$ and $p\mathrm{I}_d$ across the interelements of $\cT_h$ gives 
	\begin{align*}
		\begin{split}
			B_h((\bsig,p), (\underline{\bv}_{h}, \underline{\hat{\bv}}_{h} )) &= 
			\inner*{\bsig - \alpha p \mathrm{I}_d, \beps(\bv_{s,h})}_{\cT_h} 
			- \dual*{(\bsig - \alpha p \mathrm{I}_d)\bn, \bv_{s,h} - \hat{\bv}_{s,h}  }_{\partial \mathcal{T}_h} \\
			&\qquad - \inner*{p, \div\bv_{f,h} }_{\cT_h}  
			+ \dual*{p\bn, \bv_{f,h} - \hat{\bv}_{f,h}  }_{\partial \mathcal{T}_h} \\
			& = \inner*{\bsig - \alpha p \mathrm{I}_d, \nabla\bv_{s,h}}_{\cT_h} 
			- \dual*{(\bsig - \alpha p \mathrm{I}_d)\bn, \bv_{s,h}  }_{\partial \mathcal{T}_h} \\
			&\qquad - \inner*{p, \div\bv_{f,h} }_{\cT_h}  
			+ \dual*{p\bn, \bv_{f,h}   }_{\partial \mathcal{T}_h}.
		\end{split}
		\end{align*}
		Applying an elementwise integration by parts to the right-hand side of the previous identity, followed by the substitution $ \big[ -\bdiv(\bsig - \alpha p \mathrm{I}_d) \mid  \nabla p  \big]  = \big[\bF_s \mid \bF_f - \beta \bu_f\big] - \dot{\underline{\bu} } R$ yields 
		\begin{align}\label{b2} 
			\begin{split}
				B_h((\bsig,p), (\underline{\bv}_{h}, \underline{\hat{\bv}}_{h} )) &= - \inner{\bdiv(\bsig - \alpha p \mathrm{I}_d), \bv_{s,h}}_{\cT_h} + \inner{\nabla p, \bv_{f,h}}_{\cT_h} 
				\\&
				= -\inner{\dot{\underline{\bu}} R, \underline{\bv}_h }_{\Omega} - \inner{\beta \bu_f, \bv_{f,h}}_{\Omega} + \inner{\underline{\bF}(t), \underline{\bv}_{h}}_{\Omega}.	
			\end{split} 
		\end{align}
	
		On the other hand,  we have that 
		\begin{equation}\label{b1}
			B_h((\btau_h,q_h), (\bu, \bu|_{\cF_h})) = \inner*{\btau_h - \alpha q_h \mathrm{I}_d, \beps(\bu_s) }_{\cT_h} -  
			\inner*{q_h, \div\bu_f}_{\cT_h} \quad \forall (\btau_h, q_h) \in \cH_{2,h}.
		\end{equation}
		Substituting back \eqref{b2} and \eqref{b1} in \eqref{consistent}  gives  
		\begin{align}\label{consistent0}
			\begin{split}
				\inner*{\dot{\underline{\bu}}, \underline{\bv}_h}_{\cH_1} &+ 
				\inner{(\dot\bsig, \dot p ), (\btau_h,q_h) }_{\cH_2} + \inner{\beta \bu_{f},\bv_{f,h}}_\Omega + B_h((\bsig,p), (\underline{\bv}_{h}, \underline{\hat{\bv}}_{h} )) 
				\\& \quad 
				- B_h((\btau_h,q_h), (\underline{\bu}, \underline{\bu}|_{\partial \cF_h} )) 
				 + \dual{ \tfrac{(k+1)^2}{h_\cF}(\underline{\bu} - \underline{\bu}|_{\partial \cF_h}) , \underline{\bv}_{h} - \underline{\hat{\bv}}_{h} }_{\partial \cT_h}  
				 \\&
				 = \inner{(\dot\bsig, \dot p ), (\btau_h,q_h) }_{\cH_2} 
				 - \inner*{ \btau_h - \alpha q_h \mathrm{I}_d, \beps(\bu_s) }_{\cT_h} + \inner*{q_h, \div\bu_f}_{\cT_h} + \inner{\underline{\bF}(t), \underline{\bv}_{h}}_{\Omega},
			\end{split}
		\end{align}
		for all $(\btau_h, q_h)\in \mathcal H_{2,h}$ and $(\underline{\bv}_h , \hat{\underline{\bv} }_h)\in \cH_{1,h} \times \hat{\cH}_{1,h}$. Moreover, applying Green's formula \eqref{green2} in the second equation of \eqref{eq:weakform} and keeping in mind the density of the embedding $H(\bdiv,\Omega,\bbS)\times H^1(\Omega) \hookrightarrow \cH_{2}$ we obtain
		\begin{equation*}
			\inner{(\dot\bsig, \dot p ), (\btau,q) }_{\cH_2} = (g,q)_\Omega  + \inner*{\btau - \alpha q \mathrm{I}_d , \beps(\bu_s)}_\Omega - \inner*{q, \div \bu_f}_\Omega 
		\quad \forall (\btau,q) \in \cH_{2}.
		 \end{equation*}
		Using  this identity in \eqref{consistent0} gives the result.
\end{proof}

\section{Convergence analysis of the HDG method}\label{sec:convergence}

The convergence analysis of the HDG method \eqref{sd} follows standard procedures. Using the stability of the HDG method and the consistency result \eqref{consistent}, we prove that the projected errors 
\begin{align*}
    \begin{alignedat}{2}
        \be_{u_s, h}(t)      &:= \Pi^{k+1}_\cT \bu_s  - \bu_{s,h}, &
		\quad 
		\be_{u_f,h}(t)       &:= \Pi^{k+1}_\cT \bu_f  - \bu_{f,h}, \\
        \be_{\hat u_s, h}(t) &:= \Pi_\cF^{k+1}(\bu_s|_{\cF_h}) - \hat{\bu}_{s,h}, &\quad    
        \be_{\hat u_f, h}(t) &:= \Pi_\cF^{k+1}(\bu_f|_{\cF_h}) - \hat{\bu}_{f,h}, \\
        \be_{\sigma,h}(t)    &:= \Pi_\cT^k\bsig - \bsig_h, &\quad 
        e_{p,h}(t)         &:= \Pi_\cT^kp - p_h, 
    \end{alignedat}
\end{align*}
can be estimated in terms of the approximation errors
\begin{align*}
    \begin{alignedat}{2}
        \bchi_{u_s}(t)       &:= \bu_s  - \Pi^{k+1}_\cT \bu_s, &\quad 
        \bchi_{u_f}(t)       &:= \bu_f  - \Pi^{k+1}_\cT \bu_f, \\
        \bchi_{\hat u_s}(t)  &:= \bu_s|_{\cF_h} - \Pi_\cF^{k+1}(\bu_s|_{\cF_h}), &\quad 
        \bchi_{\hat u_f}(t)  &:= \bu_f|_{\cF_h} - \Pi_\cF^{k+1}(\bu_f|_{\cF_h}), \\
        \bchi_\sigma(t)      &:= \bsig - \Pi_\cT^k\bsig, &\quad 
        \chi_p(t)           &:= p - \Pi_\cT^kp.
    \end{alignedat}
\end{align*}
As before, we concatenate the error terms corresponding to velocities by defining
\begin{align*}
    \begin{alignedat}{2}
        \underline{\be}_{u,h}(t)       &:= [\be_{ u_s, h} \mid \be_{u_f, h}], &\quad 
        \underline{\be}_{\hat u,h}(t)       &:= [\be_{\hat u_s, h} \mid \be_{\hat u_f, h}], \\
        \underline{\bchi}_u(t)  &:= [ \bchi_{u_s} \mid \bchi_{u_f}], &\quad 
        \underline{\bchi}_{\hat u}(t)  &:= [\bchi_{\hat u_s}  \mid \bchi_{\hat u_f}].
    \end{alignedat}
\end{align*}

\begin{lemma}\label{stab_sd}
	Under the conditions of Proposition~\ref{consistency},  there exists a constant $C>0$ independent of $h$ and $k$ such that 
	\begin{align}\label{stab}
		\begin{split}
			&\max_{[0, T]}\norm{ \underline{\be}_{u,h}}^2_{\cH_1} + \max_{[0, T]} \norm{(\be_{\sigma,h},  e_{p,h})}^2_{\cH_2} + \int_0^T \norm{ \beta^{\sfrac12}  \be_{\bu_f,h}(s)}^2_{0,\Omega_f} \,\text{d}s
			\\& 
			+ \int_0^T \norm{\tfrac{k+1}{h_\cF^{\sfrac{1}{2}}}(\underline{\be}_{u,h} - \underline{\be}_{\hat{u}, h} )}^2_{0,\partial \cT_h}\,\text{d}s
 		 \leq C \int_0^T \Big(\norm{\tfrac{h_\cF^{\sfrac{1}{2}}}{k+1} \bchi_\sigma}^2_{0,\partial \cT_h}  +  \norm{\tfrac{h_\cF^{\sfrac{1}{2}}}{k+1} \chi_p}^2_{0,\partial \cT_h}  +  \abs{ (\underline{\bchi}_u,\underline{\bchi}_{\hat u})}^2_{\mathcal{U} \times \hat{\mathcal{U}}} \Big) \, \text{d}t.
		\end{split}
	\end{align} 
	\end{lemma}
	\begin{proof}
	By virtue of the consistency result \eqref{consistent},  it is straightforward that 
	 \begin{align}\label{orthog} 
		\begin{split}
			& \inner*{\dot{\underline{\be} }_{u,h}, \underline{\bv}}_{\cH_1} + 
			\inner{(\dot{\be}_{\sigma,h}, \dot{\be}_{p,h}), (\btau,q) }_{\cH_2}  + \inner{\beta \be_{u_f,h},\bv_f}_\Omega
			+ B_h((\be_{\sigma,h}, e_{p,h}), (\underline{\bv}, \underline{\hat{\bv}} )) 
			\\ 
			&\qquad \qquad  - B_h((\btau,q), (\underline{\be}_{u,h}, \underline{\be}_{\hat{u},h} )) 
			 +\dual{ \tfrac{(k+1)^2}{h_\cF}(\underline{\be}_{u,h} - \underline{\be}_{\hat{u}, h} ), \underline{\bv} - \underline{\hat{\bv}} }_{\partial \cT_h}
			 \\ &
			  = 
			  -\inner*{\dot{\underline{\bchi} }_{u}, \underline{\bv}}_{\cH_1} - 
			\inner{(\dot{\bchi}_{\sigma}, \dot{\bchi}_{p}), (\btau,q) }_{\cH_2}  - \inner{\beta \bchi_{u_f},\bv_f}_\Omega
			- B_h((\bchi_{\sigma}, \bchi_{p}), (\underline{\bv}, \underline{\hat{\bv}} )) 
			\\ &
			\qquad \qquad   + B_h((\btau,q), (\underline{\bchi}_{u}, \underline{\bchi}_{\hat{u}} )) 
			 - \dual{ \tfrac{(k+1)^2}{h_\cF}(\underline{\bchi}_{u} - \underline{\bchi}_{\hat{u}} ), \underline{\bv} - \underline{\hat{\bv}} }_{\partial \cT_h}
			 \\ &
			 = 
			\dual{(\bchi_\sigma - \alpha \chi_p \mathrm{I}_d)\bn, \bv_s - \hat{\bv}_s}_{\partial \cT_h} - \dual{\chi_p\bn, \bv_f - \hat{\bv}_f}_{\partial \cT_h}
			\\ &
			\qquad \qquad +   B_h((\btau,q), (\underline{\bchi}_{u}, \underline{\bchi}_{\hat{u}} )) 
			 - \dual{ \tfrac{(k+1)^2}{h_\cF}(\underline{\bchi}_{u} - \underline{\bchi}_{\hat{u}} ), \underline{\bv} - \underline{\hat{\bv}} }_{\partial \cT_h}
		\end{split} 
	\end{align}
	for all $(\btau,q) \in \mathcal H_{2,h}$ and $(\underline{\bv} , \hat{\underline{\bv}} )\in \mathcal{H}_{1,h} \times \hat{\mathcal{H}}_{1,h}$. 
The last identity is derived from the orthogonality properties 
 	\begin{align*}
		\inner{\dot{\underline{\bchi} }_{u}, \underline{\bv}}_{\cH_1} + \inner{\beta \bchi_{u_f},\bv_f}_\Omega =0 \quad \forall \underline{\bv} \in \cH_{1,h}, 
	\quad \text{and} \quad 
		\inner{(\dot{\bchi}_{\sigma}, \dot{\bchi}_{p}), (\btau,q) }_{\cH_2}   =0 \quad \forall (\btau,q)\in \cH_{2,h}, 
	\end{align*}
 	and from the fact that, for all   $(\underline{\bv} , \hat{\underline{\bv}} )\in \cH_{1,h} \times \hat{\cH}_{1,h}$, the following holds: 
 	\[
		B_h((\bchi_{\sigma}, \bchi_{p}), (\underline{\bv}, \underline{\hat{\bv}} )) = -\dual{(\bchi_\sigma - \alpha \chi_p \mathrm{I}_d)\bn, \bv_s - \hat{\bv}_s}_{\partial \cT_h} +  \dual{\chi_p\bn, \bv_f - \hat{\bv}_f}_{\partial \cT_h}.
 	\]
 	This result is further supported by the inclusions:
    \[
    \beps(\cP_{k+1}(\cT_h, \bbR^d)) \subset \cP_k(\cT_h, \bbS) \quad  \text{and} \quad  \div(\cP_{k+1}(\cT_h, \bbR^d)) \subset \cP_k(\cT_h).
    \]

The choices $\btau = \be_{\sigma,h}$, $q = e_{p,h}$, and $(\underline{\bv} ,\hat{\underline{\bv} }) = (\underline{\be}_{u,h}, \underline{\be}_{\hat{u},h} )$ in \eqref{orthog} and the Cauchy-Schwarz inequality together with \eqref{Bhh} yield 
\begin{align*}
	\frac12 \frac{\text{d}}{\text{d}t} &\big\{\norm{ \underline{\be}_{u,h}}^2_{\cH_1} 
	+ \norm{(\be_{\sigma,h},  e_{p,h})}^2_{\cH_2}  \big\} 
	+ \norm{ \beta^{\sfrac12}  \be_{\bu_f,h}}^2_{0,\Omega_f} 
	+ \norm{\tfrac{k+1}{h_\cF^{\sfrac{1}{2}}}(\underline{\be}_{u,h} - \underline{\be}_{\hat{u}, h} )}^2_{0,\partial \cT_h} \\
	&= \dual{(\bchi_\sigma - \alpha \chi_p \mathrm{I}_d)\bn, \be_{u_s,h} - \be_{\hat{u}_s,h}}_{\partial \cT_h} 
	- \dual{\chi_p\bn, \be_{u_f,h} - \be_{\hat{u}_f,h}}_{\partial \cT_h} \\
	&\quad + B_h((\be_{\sigma,h}, e_{p,h}), (\underline{\bchi}_{u}, \underline{\bchi}_{\hat{u}} )) 
	- \dual{ \tfrac{(k+1)^2}{h_\cF}(\underline{\bchi}_{u} - \underline{\bchi}_{\hat{u}} ), \underline{\be}_{u,h} - \underline{\be}_{\hat{u},h} }_{\partial \cT_h} \\
	&\leq \norm{\tfrac{h_\cF^{\sfrac{1}{2}}}{k+1} (\bchi_\sigma - \alpha \chi_p \mathrm{I}_d)\bn}_{0,\partial \cT_h} 
	\norm{\tfrac{k+1}{h_\cF^{\sfrac{1}{2}}} (\be_{u_s,h} - \be_{\hat{u}_s,h})}_{0,\partial \cT_h} \\
	&\quad + \norm{\tfrac{h_\cF^{\sfrac{1}{2}}}{k+1} \chi_p \bn}_{0,\partial \cT_h} 
	\norm{\tfrac{k+1}{h_\cF^{\sfrac{1}{2}}} (\be_{u_f,h} - \be_{\hat{u}_f,h})}_{0,\partial \cT_h} \\
	&\quad + C \norm{(\be_{\sigma,h}, e_{p,h})}_{\cH_2}  
	\abs{ (\underline{\bchi}_u,\underline{\bchi}_{\hat u})}_{\mathcal{U} \times \hat{\mathcal{U}}} \\
	&\quad + \norm{\tfrac{k+1}{h_\cF^{\sfrac{1}{2}}} (\underline{\bchi}_{u} - \underline{\bchi}_{\hat{u}})  }_{0,\partial \cT_h} 
	\norm{\tfrac{k+1}{h_\cF^{\sfrac{1}{2}}} (\underline{\be}_{u,h} - \underline{\be}_{\hat{u},h})  }_{0,\partial \cT_h}. 
\end{align*}
We notice that, because of assumption \eqref{initial-R1-R2-h*c}, the projected errors  satisfy vanishing initial conditions, namely, $\be_{\sigma,h}(0) = \mathbf 0$, $e_{p,h}(0) = 0$  and $(\underline{\be}_{u,h}(0), \underline{\be} _{\hat u,h}(0)) = (\mathbf 0, \mathbf 0)$. Hence, integrating over $t\in (0, T]$ and using again the Cauchy-Schwarz inequality we deduce that 
\begin{align*}
 		&\norm{ \underline{\be}_{u,h}}^2_{\cH_1} + \norm{(\be_{\sigma,h},  e_{p,h})}^2_{\cH_2} + \int_0^t \norm{ \beta^{\sfrac12}  \be_{\bu_f,h}(s)}^2_{0,\Omega_f} \,\text{d}s+ \int_0^t \norm{\tfrac{k+1}{h_\cF^{\sfrac{1}{2}}}(\underline{\be}_{u,h} - \underline{\be}_{\hat{u}, h} )}^2_{0,\partial \cT_h}\,\text{d}s
 		\\
 		&\qquad \lesssim \Big(\int_0^T (\norm{\tfrac{h_\cF^{\sfrac{1}{2}}}{k+1} (\bchi_\sigma - \alpha \chi_p \mathrm{I}_d)\bn}^2_{0,\partial \cT_h}  +  \norm{\tfrac{h_\cF^{\sfrac{1}{2}}}{k+1} \chi_p \bn}^2_{0,\partial \cT_h} + \norm{\tfrac{k+1}{h_\cF^{\sfrac{1}{2}}} (\underline{\bchi}_{u} - \underline{\bchi}_{\hat{u}})  }^2_{0,\partial \cT_h})  \text{d}t\Big)^{\sfrac{1}{2}}  
		\\ & \qquad \qquad \qquad \qquad \qquad \qquad \qquad \qquad \qquad \qquad \qquad
		\times 
 		\Big( \int_0^T \norm{\tfrac{k+1}{h_\cF^{\sfrac{1}{2}}} (\underline{\be}_{u,h} - \underline{\be}_{\hat{u},h})  }^2_{0,\partial \cT_h}\ \text{d}t\Big)^{\sfrac{1}{2}}
 		\\
 		&\qquad \qquad + \Big(\int_{0}^{T}\norm{(\be_{\sigma,h}, e_{p,h})}^2_{\cH_2} \text{d}t\Big)^{\sfrac12} \Big( \int_0^T \abs{ (\underline{\bchi}_u,\underline{\bchi}_{\hat u})}^2_{\mathcal{U} \times \hat{\mathcal{U}}}  \text{d}t \Big)^{\sfrac12},\quad \forall t\in (0, T].
\end{align*}
Finally, a simple application of Young's inequality yields 
\begin{align*}
 		 &\max_{[0, T]}\norm{ \underline{\be}_{u,h}}^2_{\cH_1} + \max_{[0, T]} \norm{(\be_{\sigma,h},  e_{p,h})}^2_{\cH_2} + \int_0^T \norm{ \beta^{\sfrac12}  \be_{\bu_f,h}(s)}^2_{0,\Omega_f} \,\text{d}s+ \int_0^T \norm{\tfrac{k+1}{h_\cF^{\sfrac{1}{2}}}(\underline{\be}_{u,h} - \underline{\be}_{\hat{u}, h} )}^2_{0,\partial \cT_h}\,\text{d}s
 		\\
 		&\qquad \qquad \lesssim \int_0^T \Big(\norm{\tfrac{h_\cF^{\sfrac{1}{2}}}{k+1} (\bchi_\sigma - \alpha \chi_p \mathrm{I}_d)\bn}^2_{0,\partial \cT_h}  +  \norm{\tfrac{h_\cF^{\sfrac{1}{2}}}{k+1} \chi_p \bn}^2_{0,\partial \cT_h}  +  \abs{ (\underline{\bchi}_u,\underline{\bchi}_{\hat u})}^2_{\mathcal{U} \times \hat{\mathcal{U}}} \Big) \, \text{d}t,
\end{align*}
and the result follows.
\end{proof}

As a consequence of the stability estimate \eqref{stab}, we immediately have the following convergence result for the HDG method \eqref{sd}-\eqref{initial-R1-R2-h*c}.
\begin{theorem}\label{hpConv}
Let
\[
\underline{\bu} = [\bu_s \mid \bu_f] \in \mathcal{C}_{[0,T]}^1(\cH_1) \cap \mathcal{C}^0_{[0,T]}(\cX_1)\quad \text{and} \quad (\bsig, p) \in \mathcal{C}^1_{[0,T]}(\cH_2) \cap \mathcal{C}^0_{[0,T]}(\cX_2)
\]
be the solutions of \eqref{eq:weakform}-\eqref{weakIC}. Assume that $\bsig \in \cC^0_{[0,T]}(H^{1+r}(\Omega, \mathbb{R}^{d\times d}_{\text{sym}}))$, $p \in \cC^0_{[0,T]}(H^{1+r}(\Omega))$ and $\underline{\bu}  \in \cC^0_{[0,T]}(H^{2+r}( \Omega, \bbR^{d \times 2}))$, with $r\geq 0$. Then, there exists a constant $C>0$ independent of $h$ and $k$ such that 
\begin{align*}
		 &\max_{[0, T]}\norm{ (\underline{\bu} - \underline{\bu}_h)(t)}_{\cH_1} + \max_{[0, T]} \norm{(\bsig - \bsig_h, p - p_h)(t)}_{\cH_2} + \left(  \int_0^T \norm{\tfrac{k+1}{h_\cF^{\sfrac{1}{2}}}(\underline{\bu}_h - \underline{\hat{\bu}}_{h} )}^2_{0,\partial \cT_h}\,\text{d}s \right)^{\sfrac{1}{2}}
		\\
		&\quad \leq  C \tfrac{h_K^{\min\{ r, k \}+1}}{(k+1)^{r+\sfrac12}} \Big( \max_{[0,T]}\norm{\underline{\bu} }^2_{2+r,\Omega} + \max_{[0, T]}\norm*{\btau}_{1+r, \Omega} + \max_{[0, T]}\norm*{p}_{1+r, \Omega}   \Big)\quad \forall k\geq 0. 
\end{align*}
\end{theorem}
\begin{proof}
		It follows from the triangle inequality and \eqref{stab} that 
	\begin{align*}
		&\max_{[0, T]}\norm{ (\underline{\bu} - \underline{\bu}_h)(t)}_{\cH_1} 
		+ \max_{[0, T]} \norm{(\bsig - \bsig_h, p - p_h)(t)}_{\cH_2} \\
		&\quad + \left( \int_0^T \norm{\tfrac{k+1}{h_\cF^{\sfrac{1}{2}}}
		((\underline{\bu} - \underline{\bu}_h) - ( \underline{\bu} - \underline{\hat{\bu}}_{h}) )}^2_{0,\partial \cT_h}\,\text{d}s \right)^{\sfrac{1}{2}} \\
		&\lesssim \max_{[0, T]}\norm{\underline{\bchi}_u(t)}^2_{\cH_1} 
		+ \max_{[0, T]}\norm{(\bchi_\sigma, \chi_p)(t)}^2_{\cH_2} \\
		&\quad + \left( \int_0^T \Big(\norm{\tfrac{h_\cF^{\sfrac{1}{2}}}{k+1} \bchi_\sigma}^2_{0,\partial \cT_h}  
		+ \norm{\tfrac{h_\cF^{\sfrac{1}{2}}}{k+1} \chi_p}^2_{0,\partial \cT_h}  
		+ \abs{ (\underline{\bchi}_u,\underline{\bchi}_{\hat u})}^2_{\mathcal{U} \times \hat{\mathcal{U}}} \Big) \, \text{d}t \right)^{\sfrac{1}{2}},
	\end{align*}
	and the result follows directly from the error estimates \eqref{tool1} and \eqref{tool2}.
\end{proof}
	
\begin{remark}\label{R1}
The energy norm error estimates in Theorem~\ref{hpConv} achieve quasi-optimality in $h$ but remain suboptimal in $k$ by a factor of $k^{\sfrac{1}{2}}$, a limitation noted in previous work such as Houston et al.~\cite{houston2002} for stationary second-order elliptic problems. The velocity fields error estimate is also suboptimal in $h$ by one order. While our numerical results demonstrate optimal convergence rates in practice, deriving a quasi-optimal $L^2$-norm error estimate for velocities remains an open challenge.
\end{remark}

\section{The fully discrete scheme and its convergence analysis}\label{sec:fully-discrete}

Given $L\in \mathbb{N}$, we consider a uniform partition of the time interval $[0, T]$ with step size $\Delta t := T/L$ and nodes $t_n := n\,\Delta t$, $n=0,\ldots, L$. The midpoint of each time subinterval is represented as $t_{n+\sfrac{1}{2}}:= \frac{t_{n+1} + t_n}{2}$. 

In what follows, we utilize the Crank-Nicolson method for the time discretisation of
\eqref{sd}-\eqref{initial-R1-R2-h*c}. Namely, for  $n=0,\ldots,L-1$, 
we seek  $(\underline{\bu}^{n+1}_h,  \underline{\hat{\bu}}^{n+1}_h) \in \cH_{1,h}\times \hat{\cH}_{1,h}$ and $(\bsig^{n+1}_h, p^{n+1}_h) \in  \mathcal{H}_{2,h}$ solution of 
\begin{align}\label{fd}
\begin{split}  
	&\tfrac{1}{\Delta t}\inner{ (\underline{\bu}^{n+1}_h - \underline{\bu}^n_h), \underline{\bv} }_{\cH_1} + \tfrac{1}{\Delta t}\inner{ (\bsig^{n+1}_h - \bsig^n_h, p^{n+1}_h - p^n_h), (\btau,q) }_{\cH_2}  + 
	\tfrac{1}{2}\inner{ \beta  (\bu^{n+1}_{f,h} + \bu^{n}_{f,h} ),  \bv_f}_{\Omega}
	\\& \quad  
	+ \tfrac12 B_h((\bsig^{n+1}_h + \bsig^n_h, p^{n+1}_h + p^n_h), (\underline{\bv}, \hat{\underline{\bv} }))    
	 - \tfrac12 B_h( (\btau, q) , (\underline{\bu} ^{n+1}_h + \underline{\bu} ^{n}_h, \hat{\underline{\bu} }_h^{n+1} +\hat{\underline{\bu} }_h^n ) ) 
	 \\ &\quad
	+\tfrac12 \dual{ \tfrac{(k+1)^2}{h_\cF}(\underline{\bu}^{n+1}_h + \underline{\bu}^{n}_h - \hat{\underline{\bu}}_h^{n+1} -\hat{\underline{\bu} }_h^n), \underline{\bv}  - \hat{\underline{\bv}}, }_{\partial \cT_h}
	\\ &
= \tfrac12 \inner{ \underline{\bF}(t_{n+1}) + \underline{\bF}(t_{n}), \underline{\bv}}_{\Omega} + \tfrac12 \inner{ g(t_{n+1}) + g(t_{n}), q}_{\Omega}
\end{split}
\end{align}
for all $(\underline{\bv}, \hat{\underline{\bv}})\in \cH_{1,h}\times \hat{\cH}_{1,h}$ and $(\btau,q)\in \mathcal H_{2,h}$. We assume that the scheme \eqref{fd} is initiated  with 
 \begin{equation}\label{initial-fd}
	\ubu^0_h = \Pi^{k+1}_\cT \ubu^0, \quad \hat{\ubu}^0_h  = \Pi^{k+1}_\cF (\ubu^0|_{\partial \mathcal{F}_h}), \quad \bsig^0_{h}= \Pi_\cT^k \bsig^0, \quad \text{and} \quad p^0_h = \Pi_\cT^k p^0.
 \end{equation}

We point out that each iteration step of \eqref{fd} requires solving a square system of linear equations whose matrix stems from the bilinear form 
\begin{align*}
	&\tfrac{1}{\Delta t}\inner{ \underline{\bu}, \underline{\bv} }_{\cH_1} + \tfrac{1}{\Delta t}\inner{ (\bsig, p), (\btau,q) }_{\cH_2}  + 
	\tfrac{1}{2}\inner{ \beta \bu_{f} ,  \bv_f}_{\Omega}
	\\
	& \qquad 
	+ \tfrac{1}{2} B_h((\bsig,p), (\underline{\bv},\hat{\underline{\bv}})) 
	- \tfrac{1}{2} B_h((\btau,q), (\underline{\bu} ,\hat{\underline{\bu} })) + \tfrac{1}{2}\dual{ \tfrac{(k+1)^2}{h_\cF}(\underline{\bu}  - \hat{\underline{\bu}} ), \underline{\bv}  - \hat{\underline{\bv}}}_{\partial \cT_h}.
\end{align*} 
The coerciveness of this bilinear form on $( (\cH_{1,h}\times \hat{\cH}_{1,h}) \times \cH_{2,h})\times ((\cH_{1,h}\times \hat{\cH}_{1,h}) \times \cH_{2,h})$ ensures the well-defined nature of the scheme \eqref{fd}-\eqref{initial-fd}. 

Our aim now is to obtain a fully discrete counterpart of Lemma~\ref{stab_sd}. We recall that, according to our notations, the components of the projected errors $\underline{\be}_{u,h}^n = [\be^n_{ u_s, h} \mid \be^n_{u_f, h}]\in \cH_{1,h} $, $\underline{\be}_{\hat u,h}^n= [\be^n_{\hat u_s, h} \mid \be^n_{\hat u_f, h}]\in  \hat{\cH}_{1,h}$, and $(\be_{\sigma,h}^n, e_{p,h}^n)\in \cH_{2,h}$ are expressed as 
\begin{align*}
    \begin{alignedat}{2}
        \be_{u_s, h}^n      &:= \Pi^{k+1}_\cT \bu_s(t_n)  - \bu^n_{s,h}, &\quad 
		\be_{u_f,h}^n       &:= \Pi^{k+1}_\cT \bu_f(t_n)  - \bu_{f,h}^n, \\
        \be_{\hat u_s, h}^n &:= \Pi_\cF^{k+1}(\bu_s(t_n)|_{\cF_h}) - \hat{\bu}^n_{s,h}, &\quad    
        \be_{\hat u_f, h}^n &:= \Pi_\cF^{k+1}(\bu_f(t_n)|_{\cF_h}) - \hat{\bu}_{f,h}^n, \\
        \be_{\sigma,h}^n    &:= \Pi_\cT^k\bsig(t_n) - \bsig^n_h, &\quad 
        e_{p,h}^n         &:= \Pi_\cT^k p(t_n) - p^n_h, 
    \end{alignedat}
\end{align*} 
while $\underline{\bchi}_{u,h}^n = [\bchi^n_{ u_s, h} \mid \bchi^n_{u_f, h}]$, $\underline{\bchi}_{\hat u,h}^n= [\bchi^n_{\hat u_s, h} \mid \bchi^n_{\hat u_f, h}]$, and $(\bchi_{\sigma,h}^n, \chi_{p,h}^n)$ are given by  
\begin{align*}
    \begin{alignedat}{2}
        \bchi_{u_s}^n       &:= \bu_s(t_n)  - \Pi^{k+1}_\cT \bu_s(t_n), &\quad 
        \bchi_{u_f}^n       &:= \bu_f(t_n)  - \Pi^{k+1}_\cT \bu_f(t_n), \\
        \bchi_{\hat u_s}^n  &:= \bu_s(t_n)|_{\cF_h} - \Pi_\cF^{k+1}(\bu_s(t_n)|_{\cF_h}), &\quad 
        \bchi_{\hat u_f}^n  &:= \bu_f(t_n)|_{\cF_h} - \Pi_\cF^{k+1}(\bu_f(t_n)|_{\cF_h}), \\
        \bchi_\sigma^n      &:= \bsig(t_n) - \Pi_\cT^k\bsig(t_n), &\quad 
        \chi_p^n           &:= p(t_n) - \Pi_\cT^k p(t_n).
    \end{alignedat}
\end{align*}

\begin{lemma}\label{coco}
	Let $\underline{\bu} = [\bu_s|\bu_f] \in \mathcal{C}_{[0,T]}^1(\cH_1) \cap \mathcal{C}^0_{[0,T]}(\cX_1)$ and  $(\bsig , p) \in \mathcal{C}^1_{[0,T]}(\cH_2) \cap \mathcal{C}^0_{[0,T]}(\cX_2)$ be the solutions of \eqref{eq:weakform}  and assume that $\bsig - \alpha p \mathrm{I}_d$ belongs  to  $\cC_{[0,T]}^0(H^s(\cT_h, \bbS))$ and $\underline{\bu}  \in \cC_{[0,T]}^0(H^s(\cT_h, \mathbb{R}^{d \times 2}))$, with $s>\sfrac{1}{2}$. Then, there exists a constant $C>0$ independent of $h$, $k$, and $\Delta t$ such that 
	\begin{align}\label{stabE} 
		\begin{split}
			  &\max_n\norm{ \underline{\be}^{n}_{u,h} }_{\cH_1}^2  + \max_n \norm{ (\be_{\sigma,h}^n,  e_{p,h}^n)}_{\cH_2}^2 + \tfrac{\Delta t}{4}  \sum_{n=0}^{L-1} \norm{\beta^{\sfrac12}(\be_{\bu_f,h}^{n+1} + \be_{\bu_f,h}^{n})}^2_{0,\Omega_f}
			  \\
			  &  \quad \quad + \tfrac{\Delta t}{4}  \sum_{n=0}^{L-1}\norm{ \tfrac{(k+1)}{h_\cF^{\sfrac{1}{2}}}(\underline{\be}^{n+1}_{u,h} + \underline{\be}^{n}_{u,h} - \underline{\be}_{\hat u, h}^{n+1} - \underline{\be}_{\hat u, h}^{n} ) }^2_{0,\partial \cT_h} 
			  \\
			  & \quad
			  \leq C \Big(  \Delta t\sum_{n=0}^{L-1} \norm{\underline{\Xi}^n_{\bu}}_{\cH_1}^2 + \Delta t\sum_{n=0}^{L-1} \norm{ (\Xi^n_{\bsig}, \Xi^n_{p}) }_{\cH_2}^2 + \Delta t \sum_{n=0}^{L-1}  \abs{(\bchi_u^{n+1} + \bchi_u^n, \bchi_{\hat u}^{n+1} + \bchi_{\hat u}^n )}_{\mathcal{U} \times \hat{\mathcal{U}}}^2
			 \\ 
			&  
			  \quad \quad  + \Delta t \sum_{n=0}^{L-1} \norm{\tfrac{h_\cF^{\sfrac{1}{2}}}{k+1} (\bchi_\sigma^{n+1} + \bchi_\sigma^{n})\bn}_{0,\partial \cT_h}^2 
			  + \Delta t \sum_{n=0}^{L-1} \norm{\tfrac{h_\cF^{\sfrac{1}{2}}}{k+1} (\chi_p^{n+1} + \chi_p^{n})}_{0,\partial \cT_h}^2 \Big),
		\end{split}
	\end{align}
	where the time consistency terms $\underline{\Xi}^n_{\bu}$, $\Xi^n_{\bsig}$, and $\Xi^n_{p}$, are defined as
	\[
	\begin{aligned}
		\underline{\Xi}^n_{\bu} &:= \frac{1}{\Delta t} (\underline{\bu}(t_{n+1}) - \underline{\bu}(t_n)) -  \tfrac12 ( \dot{\underline{\bu}}(t_{n+1}) + \dot{\underline{\bu}}(t_{n})), \\
		\Xi^n_{\bsig} &:= \tfrac{1}{\Delta t}(\bsig(t_{n+1}) - \bsig(t_n)) - \tfrac12 (\dot\bsig(t_{n+1}) + \dot\bsig(t_n)), \\
		\Xi^n_{p} &:= \tfrac{1}{\Delta t}(p(t_{n+1}) - p(t_n)) - \tfrac12 ( \dot p(t_{n+1}) + \dot p(t_n)).
	\end{aligned}
	\]
	\end{lemma}
	\begin{proof}
	It follows from the consistency equation \eqref{consistent} and the orthogonality properties employed in the proof of Lemma~\ref{stab_sd} that the projected errors $(\be_{\sigma,h}^n, e_{p,h}^n)\in \mathcal{H}_{2,h}$ and $(\underline{\be}_{u,h}^n, \underline{\be}_{\hat{u},h}^n)\in \mathcal{H}_{1,h} \times \hat{\mathcal{H}}_{1,h}$ satisfy the equation  
	\begin{align}\label{proE} 
	\begin{split}
		&\tfrac{1}{\Delta t}\inner{ \underline{\be}_{u,h}^{n+1} - \underline{\be}_{u,h}^n, \underline{\bv} }_{\cH_1} + \tfrac{1}{\Delta t}\inner{(\be_{\sigma,h}^{n+1}, e_{p,h}^{n+1}) - (\be_{\sigma,h}^n, e_{p,h}^n), (\btau, q) }_{\cH_2}  +  \tfrac12\inner{\beta (\be_{\bu_f,h}^{n+1} + \be_{\bu_f,h}^{n}),  \bv_f}_{\Omega}  
		\\ &\qquad
		+ \tfrac12 B_h((\be_{\sigma,h}^{n+1} + \be_{\sigma,h}^n, e_{p,h}^{n+1} + e_{p,h}^n), (\underline{\bv} ,\hat{\underline{\bv} })) 
		- \tfrac12 B_h( (\btau, q), (\underline{\be}^{n+1}_{u,h} + \underline{\be}^{n}_{u,h}, \underline{\be}^{n+1}_{\hat u, h} + \underline{\be}^n_{\hat u, h} )) 
		\\ &\qquad
		+\tfrac12 \dual{ \tfrac{(k+1)^2}{h_\cF}(\underline{\be}^{n+1}_{u,h} + \underline{\be}^{n}_{u,h} - \underline{\be}^{n+1}_{\hat u, h} - \underline{\be}^n_{\hat u, h} ), \underline{\bv}  - \hat{\underline{\bv} }}_{\partial \cT_h}
		\\
	& = \inner{\underline{\Xi}^n_{\bu}, \underline{\bv} }_{\cH_1} + \inner{(\Xi^n_{\bsig}, \Xi^n_{p}), (\btau,q)}_{\cH_2} 
		\\
		&\qquad  + \tfrac{1}{2} \dual{(\bchi_\sigma^{n+1} + \bchi_\sigma^{n} - \alpha (\chi_{p}^{n+1} + \chi_{p}^n) \mathrm{I}_d )\bn, \bv_s  - \hat{\bv}_s}_{\partial \cT_h}
		-  \tfrac{1}{2} \dual{(\chi_p^{n+1} + \chi_p^n)\bn, \bv_f - \hat{\bv}_f}_{\partial \cT_h} 
		\\ 
		&\qquad + \tfrac{1}{2} B_h((\btau,q), (\underline{\bchi}_u^{n+1} + \underline{\bchi}_u^{n}, \underline{\bchi}_{\hat u}^{n+1} + \underline{\bchi}_{\hat u}^{n} )) 
		- \tfrac{1}{2} \dual{ \tfrac{(k+1)^2}{h_\cF}(\underline{\bchi}_u^{n+1} + \underline{\bchi}_u^{n} - \underline{\bchi}_{\hat u}^{n+1} - \underline{\bchi}_{\hat u}^{n} ), \underline{\bv} - \hat{\underline{\bv} }}_{\partial \cT_h}
	\end{split}
	\end{align}
	for all $(\underline{\bv}, \hat{\underline{\bv}})\in \cH_{1,h} \times \hat{\mathcal{H}}_{1,h}$ and $(\btau,q)\in \mathcal H_{2,h}$.
	Selecting $(\btau, q) = \tfrac{1}{2} (\be_{\sigma,h}^{n+1} + \be_{\sigma,h}^n, e_{p,h}^{n+1} + e_{p,h}^n)$ and $(\underline{\bv} , \hat{\underline{\bv} }) = \tfrac{1}{2} (\underline{\be}^{n+1}_{u,h} + \underline{\be}^{n}_{u,h}, \underline{\be}^{n+1}_{\hat u, h} + \underline{\be}^n_{\hat u, h} )$ in equation \eqref{proE} and applying \eqref{Bhh} and the Cauchy-Schwartz inequality to the terms on the right-hand side we derive the estimate 
	\begin{align*}
		\begin{split}
			&\tfrac{1}{2\Delta t}\left(\norm{ \underline{\be}^{n+1}_{u,h} }_{\cH_1}^2 +  \norm{ (\be_{\sigma,h}^{n+1}, e_{p,h}^{n+1}) }_{\cH_2}^2 -  \norm{ \underline{\be}^{n}_{u,h} }_{\cH_1}^2   - \norm{ (\be_{\sigma,h}^{n}, e_{p,h}^{n}) }_{\cH_2}^2  \right) 
			\\
			 &\qquad  + \tfrac{1}{4}\norm{\beta^{\sfrac12} (\be_{\bu_f,h}^{n+1} + \be_{\bu_f,h}^n)}^2_{\Omega}  
			+ \tfrac{1}{4} \norm{ \tfrac{k+1}{h_\cF^{\sfrac{1}{2}}}(\underline{\be}^{n+1}_{u,h} + \underline{\be}^{n}_{u,h} - \underline{\be}^{n+1}_{\hat u, h} - \underline{\be}^n_{\hat u, h} )}^2_{0,\partial \cT_h}  
			\\
			&\quad  \leq \tfrac{1}{2} \norm{\underline{\Xi}^n_{\bu}}_{\cH_1} \norm{ \underline{\be}^{n+1}_{u,h} + \underline{\be}^{n}_{u,h} }_{\cH_1}  
			+ \tfrac{1}{2} \norm{(\Xi^n_{\bsig}, \Xi^n_{p})}_{\cH_2} \norm{(\be_{\sigma,h}^{n+1} + \be_{\sigma,h}^n, e_{p,h}^{n+1} + e_{p,h}^n)}_{\cH_2}
			\\
			&\qquad    + \tfrac{1}{4}  \norm{\tfrac{h_\cF^{\sfrac{1}{2}}}{k+1} (\bchi_\sigma^{n+1} + \bchi_\sigma^{n} - \alpha (\chi_{p}^{n+1} + \chi_{p}^n) \mathrm{I}_d)\bn)}_{0,\partial \cT_h} \norm{ \tfrac{k+1}{h_\cF^{\sfrac{1}{2}}} 
				(\be^{n+1}_{u_s,h} + \be^{n}_{u_s,h} - \be^{n+1}_{\hat{u}_s, h} - \be^n_{\hat{u}_s, h} )
				}_{0,\partial \cT_h}
			\\
			&\qquad
			+ \tfrac{1}{4}  \norm{\tfrac{h_\cF^{\sfrac{1}{2}}}{k+1} (\chi_p^{n+1} + \chi_p^{n} )\bn)}_{0,\partial \cT_h} \norm{ \tfrac{k+1}{h_\cF^{\sfrac{1}{2}}} 
				(\be^{n+1}_{u_f,h} + \be^{n}_{u_f,h} - \be^{n+1}_{\hat{u}_f, h} - \be^n_{\hat{u}_f, h} )
				}_{0,\partial \cT_h}
			\\
			&\qquad 
			+ \tfrac{C}{4} \norm{(\be_{\sigma,h}^{n+1} + \be_{\sigma,h}^n, e_{p,h}^{n+1} + e_{p,h}^n )}_{\cH_2} \abs{ (\underline{\bchi}_u^{n+1} + \underline{\bchi}_u^{n}, \underline{\bchi}_{\hat u}^{n+1} + \underline{\bchi}_{\hat u}^{n} ) }_{\mathcal{U} \times \hat{\mathcal{U}}} 
		  \\
			&\qquad 
		  + \tfrac{1}{4} \norm{ \tfrac{k+1}{h^{\sfrac{1}{2}}_\cF}(\underline{\bchi}_u^{n+1} + \underline{\bchi}_u^{n} - \underline{\bchi}_{\hat u}^{n+1} - \underline{\bchi}_{\hat u}^{n})}_{0,\partial \cT_h} \norm{ \tfrac{k+1}{h^{\sfrac{1}{2}}_\cF}(\underline{\be}^{n+1}_{u,h} + \underline{\be}^{n}_{u,h} - \underline{\be}^{n+1}_{\hat u, h} - \underline{\be}^n_{\hat u, h} )}_{0,\partial \cT_h}.
		\end{split}
	\end{align*}
	Summing in the index $n$ and taking into account that the projected errors vanish  identically at the  initial step we  get 
	\begin{align*}
		\begin{split}
			\max_n&\norm{ \underline{\be}^{n}_{u,h} }_{\cH_1}^2  + \max_n \norm{ (\be_{\sigma,h}^{n}, e_{p,h}^{n}) }_{\cH_2}^2 + \tfrac{\Delta t}{4}  \sum_{n=0}^{L-1} \norm{\beta^{\sfrac12} (\be_{\bu_f,h}^{n+1} + \be_{\bu_f,h}^n)}^2_{\Omega} 
			\\ & \quad 
			+ \tfrac{\Delta t}{4}  \sum_{n=0}^{L-1}\norm{ \tfrac{k+1}{h_\cF^{\sfrac{1}{2}}}(\underline{\be}^{n+1}_{u,h} + \underline{\be}^{n}_{u,h} - \underline{\be}^{n+1}_{\hat u, h} - \underline{\be}^n_{\hat u, h} )}^2_{0,\partial \cT_h} 
			\\
			 & \leq \max_n\norm{ \underline{\be}^{n}_{u,h} }_{\cH_1} \Big( \Delta t\sum_{n=0}^{L-1} \norm{\underline{\Xi}^n_{\bu}}_{\cH_1}   \Big)
			  + \max_n \norm{ (\be_{\sigma,h}^{n}, e_{p,h}^{n}) }_{\cH_2} \Big(\Delta t\sum_{n=0}^{L-1} \norm{(\Xi^n_{\bsig}, \Xi^n_{p}}_{\cH_2} \Big) 
		\\
		& \quad+ \tfrac{\Delta t}{4}  \sum_{n=0}^{L-1} \norm{\tfrac{h_\cF^{\sfrac{1}{2}}}{k+1} (\bchi_\sigma^{n+1} + \bchi_\sigma^{n} - \alpha (\chi_{p}^{n+1} + \chi_{p}^n) \mathrm{I}_d)\bn}_{0,\partial \cT_h} \norm{ \tfrac{k+1}{h_\cF^{\sfrac{1}{2}}}(\be^{n+1}_{u_s,h} + \be^{n}_{u_s,h} - \be^{n+1}_{\hat{u}_s, h} - \be^n_{\hat{u}_s, h} )}_{0,\partial \cT_h}  
		\\
		&\quad +
		\tfrac{\Delta t}{4}  \sum_{n=0}^{L-1} \norm{\tfrac{h_\cF^{\sfrac{1}{2}}}{k+1} (\chi_p^{n+1} + \chi_p^{n} )\bn}_{0,\partial \cT_h} \norm{ \tfrac{k+1}{h_\cF^{\sfrac{1}{2}}} 
		(\be^{n+1}_{u_f,h} + \be^{n}_{u_f,h} - \be^{n+1}_{\hat{u}_f, h} - \be^n_{\hat{u}_f, h} )
		}_{0,\partial \cT_h}
		\\
		&\quad +  C \Delta t \max_n \norm{ (\be_{\sigma,h}^{n}, e_{p,h}^{n}) }_{\cH_2}  \sum_{n=0}^{L-1}  \abs{ (\underline{\bchi}_u^{n+1} + \underline{\bchi}_u^{n}, \underline{\bchi}_{\hat u}^{n+1} + \underline{\bchi}_{\hat u}^{n} ) }_{\mathcal{U} \times \hat{\mathcal{U}}} 
		\\
		&\quad + \tfrac{\Delta t}{4} \sum_{n=0}^{L-1} \norm{ \tfrac{k+1}{h^{\sfrac{1}{2}}_\cF}(\bchi_u^{n+1} + \bchi_u^{n} - \bchi_{\hat u}^{n+1} - \bchi_{\hat u}^{n})}_{0,\partial \cT_h} \norm{ \tfrac{k+1}{h^{\sfrac{1}{2}}_\cF}(\be^{n+1}_{u,h} + \be^{n}_{u,h} - \be^{n+1}_{\hat u, h} -  \be^n_{\hat u, h} )}_{0,\partial \cT_h}.
		\end{split}
	\end{align*}
	Applying the Cauchy-Schwartz inequality along with  Young's inequality $2a b \leq \frac{a^2}{\epsilon} +  \epsilon b^2$, where a suitable $\epsilon > 0$ is chosen in each instance, yields \eqref{stabE}. 
	\end{proof}
	
	To handle the time-consistency terms in \eqref{stabE}, we rely on a Taylor expansion centered at $t=t_{n+\sfrac{1}{2}}$, yielding the identity:
	\[
	  \tfrac{1}{\Delta t} (\varphi(t_{n+1}) - \varphi(t_n)) = \tfrac{1}{2} (\dot\varphi(t_{n+1}) + \dot\varphi(t_n)) + \tfrac{(\Delta t)^2}{16} \int_{-1}^1 \dddot{\varphi} (t_{n+\sfrac{1}{2}} + \tfrac{\Delta t}{2} s) (|s|^2 -1) \,  \text{d}s\quad \forall \varphi \in \cC^3([0, T]).
	\]
	Therefore, under the additional assumptions $\underline{\bu}  \in \cC^{3}_{[0,T]}(\cH_1)$ and $(\bsig, p) \in \cC^{3}_{[0,T]}(\cH_2)$, we deduce from \eqref{stabE} that 
	\begin{align}\label{stabET}
		\begin{split}
			  &\max_n\norm{ \underline{\be}^{n}_{u,h} }_{\cH_1}^2  + \max_n \norm{ (\be_{\sigma,h}^{n}, e_{p,h}^{n}) }_{\cH_2}^2 + \tfrac{\Delta t}{4}  \sum_{n=0}^{L-1} \norm{\beta^{\sfrac12} (\be_{\bu_f,h}^{n+1} + \be_{\bu_f,h}^n)}^2_{\Omega} 
			  \\ & \quad \quad 
			  + \tfrac{\Delta t}{4}  \sum_{n=0}^{L-1}\norm{ \tfrac{k+1}{h_\cF^{\sfrac{1}{2}}}(\underline{\be}^{n+1}_{u,h} + \underline{\be}^{n}_{u,h} - \underline{\be}^{n+1}_{\hat u, h} - \underline{\be}^n_{\hat u, h} )}^2_{0,\partial \cT_h} 
			  \\
			&  \quad
			\lesssim (\Delta t)^4 \big( \max_{[0,T]}\norm{\dddot{\underline{\bu} }}^2_{\cH_1} + \max_{[0,T]}\norm{(\dddot{\bsig}, \dddot{p})}^2_{\cH_2} \big)
			\\
			&
			\quad \quad  + \Delta t \sum_{n=0}^{L-1} \norm{\tfrac{h_\cF^{\sfrac{1}{2}}}{k+1} (\bchi_\sigma^{n+1} + \bchi_\sigma^{n})\bn}_{0,\partial \cT_h}^2 
			+ \Delta t \sum_{n=0}^{L-1} \norm{\tfrac{h_\cF^{\sfrac{1}{2}}}{k+1} (\chi_p^{n+1} + \chi_p^{n})}_{0,\partial \cT_h}^2
			\\
		  &  
			\quad \quad + \Delta t \sum_{n=0}^{L-1}  \abs{(\bchi_u^{n+1} + \bchi_u^n, \bchi_{\hat u}^{n+1} + \bchi_{\hat u}^n )}_{\mathcal{U} \times \hat{\mathcal{U}}}^2.
		\end{split}
	\end{align}

	We can now state the convergence result for the fully discrete scheme \eqref{fd}-\eqref{initial-fd}.
	\begin{theorem}\label{hpConvFD}
	Assume that the solution $(\ubu, (\bsig, p))$ of \eqref{eq:weakform}-\eqref{weakIC}  satisfies the time regularity assumptions $\underline{\bu}  \in \cC^3_{[0,T]}(\cH_1)$, $\underline{\bu}  \in \cC^0(H^{2+r}(\Omega, \bbR^{d \times 2}))$, $\bsig \in \cC^0(H^{1+r}(\Omega, \bbS))$, and $p \in \cC^0(H^{1+r}(\Omega))$, with $r\geq 0$. Then, there exists a constant $C>0$ independent of $h$ and $k$ such that, for all $k\geq 0$,  
	\begin{align*}
		 \max_n\norm{\underline{\bu} (t_n) - \underline{\bu}_h^n}_{\cH_1}  &+
         \max_n\norm{ (\bsig, p)(t_n) - (\bsig_h^n, p_h^n)}_{\cH_2}
 \leq  C(\Delta t)^2 \big(\max_{[0,T]}\norm{(\dddot{\bsig}, \dddot{p})}_{\cH} + \max_{[0,T]}\norm{\dddot{\underline{\bu} }}_{0,\cT_h}\big)
		\\
		&\quad + C \tfrac{h_K^{\min\{ r, k \}+1}}{(k+1)^{r+\sfrac{1}{2}}} \Big(\max_{[0,T]}\norm*{(\bsig,p)}_{1+r,\Omega}  + \max_{[0,T]}\norm{\underline{\bu} }_{2+r,\Omega}\Big). 
	\end{align*}
	\end{theorem}
	\begin{proof}
		Applying the triangle inequality  we deduce from \eqref{stabET} that 
	\begin{align*}
		&\max_n\norm{\underline{\bu} (t_n) - \underline{\bu}_h^n}_{\cH_1}  +
         \max_n\norm{ (\bsig, p)(t_n) - (\bsig_h^n, p_h^n)}_{\cH_2}
		\lesssim   
        \max_n\norm{\underline{\bchi}_{u,h}^n}_{\cH_1}  +
         \max_n\norm{ (\bchi_{\sigma,h}^n, \chi_{p,h}^n)}_{\cH_2}
         \\
         & \quad
        +\Delta t \sum_{n=0}^{L-1} \norm{\tfrac{h_\cF^{\sfrac{1}{2}}}{k+1} (\bchi_\sigma^{n+1} + \bchi_\sigma^{n})\bn}_{0,\partial \cT_h}^2 
		+ \Delta t \sum_{n=0}^{L-1} \norm{\tfrac{h_\cF^{\sfrac{1}{2}}}{k+1} (\chi_p^{n+1} + \chi_p^{n})}_{0,\partial \cT_h}^2
		\\
	  &  
		\quad \quad + \Delta t \sum_{n=0}^{L-1}  \abs{(\bchi_u^{n+1} + \bchi_u^n, \bchi_{\hat u}^{n+1} + \bchi_{\hat u}^n )}_{\mathcal{U} \times \hat{\mathcal{U}}}^2
		 + C(\Delta t)^2 \big(\max_{[0,T]}\norm{(\dddot{\bsig}, \dddot{p})}_{\cH} + \max_{[0,T]}\norm{\dddot{\underline{\bu} }}_{0,\cT_h}\big),
		\end{align*}
		and the result is a direct consequence of the error estimates \eqref{tool1} and \eqref{tool2}.
	\end{proof}

\section{Numerical results}\label{sec:numresults}

The numerical results presented in this section have been implemented using the finite element library \texttt{Netgen/NGSolve} \cite{schoberl2014c++}. Firstly, we confirm the accuracy of our HDG scheme by analyzing a problem with a manufactured solution. Then, we consider a practical model problem inspired by \cite{antoniettiIMA, morency}.

\subsection{Example 1: Validation of the convergence rates} 

In this example, we confirm the decay of error as predicted by Theorem~\ref{hpConvFD} with respect to the parameters $h$, $\Delta t$ and $k$.  We employ successive levels of refinement on an unstructured mesh and compare the computed solutions to an exact solution of problem \eqref{4field-a}-\eqref{4field-d} given by 
\begin{align}\label{exactSol}
\begin{split}
p(x,y,t) &:= \sin(\pi x y) \cos(t) \quad \text{in $\Omega\times (0, T]$},
\\
\bd(x,y,t) &:= \begin{pmatrix}
	x \cos(\pi y) \cos(t)
	\\
	y \sin(\pi x) \sin(t)
	\end{pmatrix} \quad \text{in $\Omega\times (0, T]$},
\end{split}
\end{align}
where $\Omega = (0,1)\times (0, 1)$. We assume that the solid medium is isotropic and let the constitutive law \eqref{Constitutive:solid}  be  given in terms of the tensor 
\begin{equation}\label{eq:hooke} 
\cC \btau := 2 \mu\boldsymbol{\btau} + \lambda \tr(\boldsymbol{\btau}) I,
\end{equation}
where $\mu>0$ and $\lambda>0$ are the Lam\'e coefficients. We choose $\bF_f = \mathbf 0$ and compute the source terms $\bF_s$ and $g$ corresponding to the manufactured solution \eqref{exactSol} of \eqref{4field-a}-\eqref{4field-d} with the material parameters given by \eqref{L1} or \eqref{L2}. We prescribe non-homogeneous boundary conditions \eqref{BC} with $\Gamma^s_D = \partial \Omega$ and $\Gamma^f_N=\partial \Omega$.

In our first test, the parameters are chosen as 
\begin{gather}\label{L1}
	\rho_{11} = 10, \quad \rho_{12} = 10, \quad \rho_{22} = 20, \quad \mu= 50, \quad \lambda = 100, \quad s = 1, \quad \beta = 1, \quad \alpha = 1.
\end{gather}
\begin{table}[!htb]
	\centering
	\begin{minipage}{0.47\linewidth}
		\centering
{\footnotesize
\begin{tblr}{hline{1,22} = {1.5pt,solid}, hline{2, 7, 12, 17} = {1pt, solid},
	hline{6, 11, 16, 21} = {dashed}, vline{ 3} = {dashed}, colspec = {c l  c c},}  
	 \textbf{$k$} & \textbf{$h$} & $\mathtt{e}^{L}_{hk}(\bsig, p)$ & $\mathtt{e}^{L}_{hk}(\underline{\bu})$ 
	 \\ 
	 \SetCell[r=4]{c} 0 
     & 1/16 &  2.51e+00   &  5.78e-01    \\
    & 1/32 &  9.75e-01    & 1.90e-01     \\
    & 1/64  & 3.53e-01    &  5.24e-02    \\
    & 1/128 & 1.34e-01    &  1.36e-02    \\
	 rates &   &   $1.40$ & $1.80$       \\ 
	 \SetCell[r=4]{c} 1 
    &1/16 &  2.77e-02   &  3.79e-03    \\
    &1/32 &  6.49e-03   &  5.27e-04   \\
    &1/64 &  1.41e-03   &  7.11e-05   \\
    &1/128 & 3.09e-04   &  1.05e-05   \\ 
	rates  &  & $2.16$ & $2.83$ \\ 
	 \SetCell[r=4]{c} 2 
	& 1/8  &  3.85e-03   &  6.19e-04    \\
    & 1/16 &  3.35e-04   &  2.74e-05   \\
    & 1/32 &  3.79e-05   & 1.80e-06   \\
    & 1/64 &  4.30e-06   & 1.21e-07   \\
	rates  &  & $3.26$ & $4.11$ \\ 
	 \SetCell[r=4]{c} 3 
	& 1/4  &  1.22e-03   &  2.94e-04   \\ 
    & 1/8  &  7.33e-05   &  9.24e-06   \\
    & 1/16 &  3.19e-06   &  1.71e-07   \\
    & 1/32 &  1.99e-07   &  5.13e-09   \\ 
	rates &  & $4.19$ & $5.27$ \\
\end{tblr}
}	
\caption{Error progression and convergence rates are shown for a sequence of uniform refinements in space and over-refinements in time. The errors are measured at $T=0.3$, by employing the set of coefficients \eqref{L1}. The exact solution is given by \eqref{exactSol}.}
		\label{T1}
\end{minipage}\hfill
\begin{minipage}{0.47\linewidth}
	\centering
{\footnotesize 
\begin{tblr}{hline{1,22} = {1.5pt,solid}, hline{2, 7, 12, 17} = {1pt, solid},
	hline{6, 11, 16, 21} = {dashed}, vline{ 3} = {dashed}, colspec = {c l  c c},} 
	 \textbf{$k$} & \textbf{$h$} & $\mathtt{e}^{L}_{hk}(\bsig, p)$ & $\mathtt{e}^{L}_{hk}(\underline{\bu})$ \\ 
     \SetCell[r=4]{c} 2
	& 1/8  &  8.48e+02   &  2.79e+02   \\
    & 1/16 &  7.05e+01   &  1.14e+01   \\
    & 1/32 &  8.99e+00   &  7.42e-01   \\
    & 1/64 &  1.11e+00   &  4.49e-02   \\
	rates &  & $3.19$ & $4.20$ \\ 
	 \SetCell[r=4]{c} 3
	& 1/4  &  3.55e+02   &  1.40e+02   \\
    & 1/8  &  2.18e+01   &  4.35e+00   \\
    & 1/16 &  8.41e-01   &  8.12e-02   \\
    & 1/32 &  4.77e-02   &  2.31e-03   \\
	 rates &  & $4.28$ & $5.29$ \\ 
	 \SetCell[r=4]{c} 4
	& 1/4  &  7.24e+00   &  2.11e+00   \\
    & 1/8  &  2.24e-01   &  2.37e-02   \\
    & 1/16 &  2.83e-03   &  1.72e-04   \\
    & 1/32 &  8.98e-05   &  4.63e-06   \\ 
	rates  &  & $5.43$ & $6.26$ \\ 
	\SetCell[r=4]{c} 5
	& 1/2  &  7.31e+01   &  4.81e+01   \\
    & 1/4  &  4.87e-01   & 1.20e-01    \\
    & 1/8  &  7.16e-03   &  8.54e-04   \\
    & 1/16 &  5.99e-05   &  5.03e-06   \\
	rates  &  & $6.75$ & $7.73$ \\ 
\end{tblr}
}
\caption{Error progression and convergence rates are shown for a sequence of uniform refinements in space and over-refinements in time. The errors are measured at $T=0.3$, by employing the set of coefficients \eqref{L2}. The exact solution is given by \eqref{exactSol}.}
\label{T2}
\end{minipage}
\end{table}

The interval $[0, T]$ is divided uniformly into subintervals of length $\Delta t$. Given that the Crank-Nicolson method has an error of $O(\Delta t^2)$, we set $\Delta t \approx O(h^{(k+2)/2})$ to ensure that the time discretization error does not affect the convergence order of the spatial discretization. For the tables and figures intended for accuracy verification, we denote the $L^2-$norms of the errors as follows:  
\begin{equation*}
\mathtt{e}^{L}_{hk}(\bsig,p) := \norm{(\bsig(T), p(T)) - (\bsig_h^L, p_h^L)}_{\cH_2},  \qquad \mathtt{e}^L_{hk}(\bu)  := \norm{\underline{\bu}(T)  - \underline{\bu}^L_h}_{\cH_1}.
\end{equation*}
 The rates of convergence in space are computed as 
\begin{equation}\label{rate}
	\mathtt{r}_{hk}^L(\star)  =\log(\mathtt{e}^L_{hk}(\star)/\tilde{\mathtt{e}}^L_{hk}(\star))[\log(h/\tilde{h})]^{-1}
\quad \star \in \set{(\bsig,p), \underline{\bu}},
\end{equation}
where $\mathtt{e}_{hk}^L(\star)$, $\tilde{\mathtt{e}}^L_{hk}(\star)$ denote errors generated at time $T$ on two consecutive  meshes of sizes $h$ and~$\tilde{h}$, respectively.  

In Table~\ref{T1} we present the errors, at the final time $T = 0.3$ relative to the mesh size $h$ for four different polynomial degrees $k$. We also include the arithmetic mean of the experimental convergence rates obtained by \eqref{rate}. We observe that the convergence in the stress and pressure fields achieves the optimal rate of $O(h^{k+1})$. Furthermore, Table~\ref{T1} highlights a convergence rate of $O(h^{k+2})$ for the velocities; see Remark~\ref{R1}. 

\begin{table}[h!]
	\begin{minipage}{0.47\linewidth}
	  \centering
	  {\footnotesize
	  \begin{tblr}{hline{1,7} = {1.5pt,solid},
		  hline{2} = {dashed},
		  vline{2} = {dashed},
		  colspec = {lX[c]X[c]X[c]X[c]},}
		  $\Delta t$ & $\mathtt{e}^{L}_{hk}(\bsig,p)$ & $\mathtt{r}_{hk}^L(\bsig,p)$ & $\mathtt{e}_{hk}^L(\underline{\bu})$ & $\mathtt{r}_{hk}^L(\underline{\bu})$ \\
		  1/16 &  3.57e-03 &  * &  1.13e-04 &  * \\ 
        1/32 &  9.17e-04 &  1.96 &  2.92e-05 &  1.96  \\
        1/64 &  2.33e-04 &  1.98 &  6.39e-06 &  2.19  \\
        1/128 & 5.85e-05 &  1.99 &  1.63e-06 &  1.97  \\
	  \end{tblr}
	  }
	  \captionof{table}{Computed errors for a sequence of uniform refinements in time with $h=1/16$ and $k=5$. The errors are measured at $t=1$, with the coefficients \eqref{L1}. The exact solution is provided by \eqref{exactSol}.}
	  \label{T3}
	\end{minipage}
	\hfill
	\begin{minipage}{0.47\linewidth}
	  \centering
\includegraphics[width=0.9\textwidth]{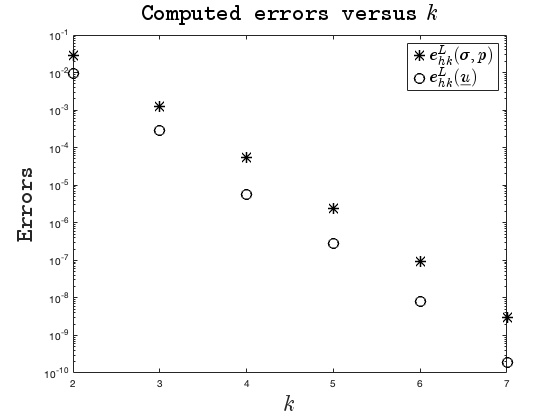}
	  \captionof{figure}{Computed errors versus the polynomial degree $k$ with $h=1/4$ and $\Delta t = 10^{-6}$. The errors are measured at $t=0.3$, by employing the coefficients \eqref{L1}. The exact solution is provided by \eqref{exactSol}.}
	  \label{T4}
	\end{minipage}
  \end{table} 

To verify the accuracy and stability of the scheme for nearly incompressible poroelastic media, we repeat the same experiment with the following set of material parameters: 
\begin{gather}\label{L2}
\rho_{11} = 10, \quad \rho_{12} = 10, \quad \rho_{22} = 20, \quad \mu= 50, \quad \lambda = 10^{8}, \quad s = 10^{-4}, \quad \beta = 1, \quad \alpha = 1.
\end{gather}
The error decay for this case is collected in Table~\ref{T2}.  These results demonstrate the ability of the proposed HDG scheme to produce accurate approximations in the case of large parameters $\lambda$.

On the other hand, Table~\ref{T3} shows the convergence results obtained after fixing the mesh size at $h=1/16$ and the polynomial degree at $k= 5$ and varying the time step $\Delta t$ used to subdivide the time interval $[0,T]$ uniformly, with $T = 1$. The convergence rates in time are calculated as
\[
\mathtt{r}_{hk}^L(\star)  =\log(\mathtt{e}^L_{hk}(\star)/\tilde{\mathtt{e}}^L_{hk}(\star))[\log(\Delta t/\widetilde{\Delta t})]^{-1}\quad \star \in \set{(\bsig,p), \underline{\bu}},
\]
where $\mathtt{e}^L_{hk}$, $\tilde{\mathtt{e}}^L_{hk}$ denote errors generated on two consecutive runs considering time steps $\Delta t$ and~$\widetilde{\Delta t}$, respectively. 
In this example, we consider the same manufactured solution obtained from \eqref{exactSol} with material coefficients \eqref{L1}. The expected convergence rate of $O(\Delta t^2)$ is reached as the time step is refined.

Finally, we fix the space mesh size $h = 1/4$ and the time mesh size $\Delta t = 10^{-6}$ and let $k$ vary from 2 to 7. In Figure~\ref{T4} we report the error $\mathtt{e}^{L}_{hk}(\bsig,p)$ in the stress variable and the error $\mathtt{e}^{L}_{hk}(\underline{\bu})$ in velocity at  $T = 0.3$  as a function of the polynomial degree $k$ on a semi-logarithmic scale. As expected, exponential convergence is observed.

\subsection{Example 2: Wave propagation in a poroelastic medium}

We investigate a wave propagation problem in a homogeneous and isotropic poroelastic medium. The computational domain $\Omega = (0,4800)\times (0,4800)\, \unit{m^2}$ contains an explosive source at position $\boldsymbol{x}_s = (1600,2900) \, \unit{m}$. We model this point-supported excitation, by defining spatial and temporal source functions by
\[
  \bF(x,y) = \begin{cases}
		\big(1 - \frac{\norm{\boldsymbol{r}}^2}{4h^2}\big)\frac{\boldsymbol{r}}{\norm{\boldsymbol{r}}} & \text{if $\norm{\boldsymbol{r}} < 2h$}
		\\
		0 & \text{otherwise}
	\end{cases}\quad \text{and} \quad 
	S(t) := (1 - 2\omega^2(t - t_0)^2)e^{-\omega^2(t - t_0)^2},
\]
where $\omega:= \pi f_0$, with peak frequency $f_0 = 5\, \unit{Hz}$, while $t_0 = 0.3\, \unit{s}$ represents the time shift parameter. The vector $\boldsymbol{r} = (x - 1600, y - 2900)^\texttt{t}$ denotes the distance from the source position, and $h$ corresponds to the mesh size used for spatial discretization. We point out that the spatial function $\bF$ creates a smooth, radially symmetric force distribution around the source point, whereas the temporal function $S(t)$ generates a Ricker wavelet with controlled frequency content. 

In the governing equations \eqref{4field-a}-\eqref{4field-d}, we set:
\begin{equation}\label{sourceW}
\bF_s = \bF_f = \bF(x,y)S(t) \quad \text{and}\quad g = 0,     
\end{equation}
and consider vanishing initial conditions. For the boundary conditions, we consider two distinct cases. On the top boundary  $\Gamma_{top} = \{(x, 4800),\quad 0 < x < 4800\}$, we impose free surface conditions: 
\[
(\bsig - \alpha p \mathrm{I}_d)\bn = \mathbf 0 \quad \text{and} \quad \bu_f\cdot \bn = 0 \quad \text{on $\Gamma_{top}\times (0, T]$}
\]
On the remaining three edges $\Gamma_{a}:= \partial \Omega\setminus \Gamma_{top}$, we implement first-order absorbing boundary conditions to prevent artificial reflections and simulate an unbounded medium (see  \cite{antoniettiIMA, morency}): 
\begin{align}\label{bcW}
\begin{split}
(\bsig - \alpha p \mathrm{I}_d)\bn &= c_{pI} \rho_{11}  (\bu_s\cdot \bn) \bn + c_{pII} \rho_f  (\bu_f\cdot \bn) \bn + (\rho_{11} - \rho_f\phi/\nu ) c_s (\mathrm{I}_d - \bn \bn^\mt) \bu_s \quad \text{on $\Gamma_{a}\times (0, T]$}
\\
-p\bn &= c_{pII} \rho_f \phi/\nu  (\bu_f\cdot \bn) \bn + c_{pI} \rho_f  (\bu_s\cdot \bn) \bn \quad \text{on $\Gamma_{a}\times (0, T]$}.
\end{split}
\end{align}
Here, $c_s:= \sqrt{\sfrac{\mu}{\rho_s}}$ represents the shear wave velocity, while $c_{pI} := \sqrt{\gamma_2}$ and $c_{pII} := \sqrt{\gamma_1}$ denote the fast and slow compressional wave velocities, respectively. The values $0 < \gamma_1 < \gamma_2$ are obtained as eigenvalues of the generalized problem $B\bw= \gamma R\bw$ where
\[
B: = \begin{pmatrix}
    \lambda + 2 \mu + \alpha^2/s & \alpha/s
    \\
    \alpha/s & 1/s
\end{pmatrix}.
\]
 
\begin{figure}[t!]
\begin{center}
\includegraphics[width=0.32\textwidth]{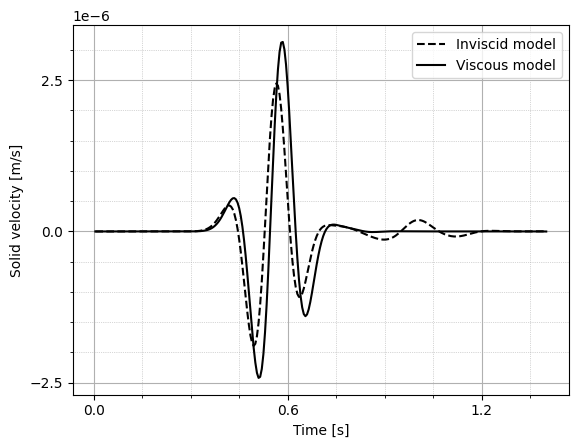}
\includegraphics[width=0.32\textwidth]{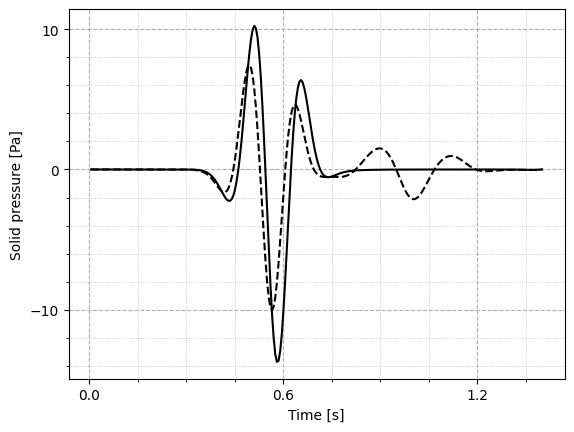}
\includegraphics[width=0.32\textwidth]{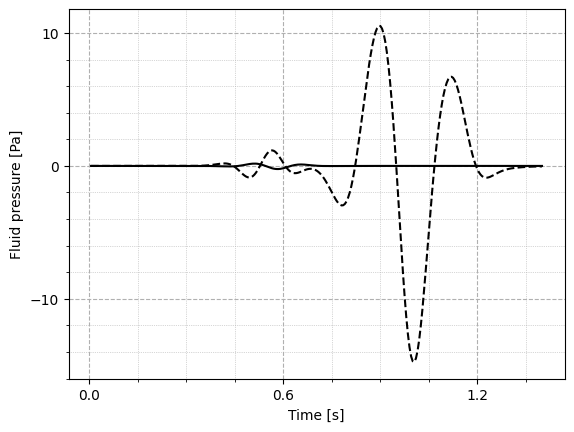}
\end{center}
\vspace{-4mm}
\caption{
The figure illustrates transient behaviour at the source receiver  $\boldsymbol{x}_r = (2000, 2200)\, \unit{m}$, displaying the $y$-components of solid velocity (left panel), the solid pressure (center panel) and fluid pressure (right panel). Two cases are compared: a purely inviscid fluid ($\eta = 0$) shown with dashed lines and a viscous fluid ($\eta = 0.0015$) represented by continuous lines. Problem \eqref{4field-a}-\eqref{4field-d} is solved using the source terms \eqref{sourceW}, boundary conditions \eqref{bcW}, and parameter set \eqref{coeffs}, with $h=100$, $k=5$, and $\Delta t = 0.005$.}
\label{fig:transients}
\end{figure}

The medium is characterized by the following physical parameters:
\begin{align}\label{coeffs}
\begin{split}
    \rho_S &= 2200\, \unit{kg/m^3}, \quad \rho_F = 950\, \unit{kg/m^3}, \quad \phi = 0.4,\quad \nu = 2, 
\\
\lambda &= \num{7.2073e9}\, \unit{Pa}, \quad \mu  = \num{4.3738e9}\, \unit{Pa},
\quad \eta = 0\, (\text{or 0.0015})\, \unit{Pa\times s}
\\
\alpha &= \num{0.0290},\quad  s = \num{1.462e-10}\, \unit{Pa^{-1}},\quad \kappa = \num{e-12}\, \unit{m^2}.
\end{split}
\end{align}

We perform numerical simulations of the poroelastic system described by equations \eqref{4field-a}-\eqref{4field-d} to investigate the influence of fluid viscosity on wave propagation patterns. The computational setup employs the following parameters: a spatial mesh size $h = 100$, polynomial elements of degree $k = 5$, and a temporal discretization with time step $\Delta t = 0.005$. The source terms are specified by \eqref{sourceW}, and the boundary conditions follow \eqref{bcW}. To elucidate the effects of viscous damping, we perform a comparative analysis between two scenarios: one with a viscous fluid ($\eta = 0.0015\, \unit{Pa\times s}$) and another with an inviscid fluid ($\eta = 0\, \unit{Pa\times s}$), while maintaining all other material parameters defined in \eqref{coeffs}. 

Figure~\ref{fig:transients} presents the vertical component of solid velocity, solid pressure (defined as $p_s := -\tr(\bsig)/2$), and pore pressure $p$. The results reveal a significant attenuation of the slow compressional wave in the diffusive model. We also observe differences in the fast wave arrival times between the viscous and inviscid cases. These observations align with similar studies \cite{antoniettiIMA, morency}.

In Figures \ref{fig:velo} and \ref{fig:mises}, we present temporal snapshots comparing wave propagation patterns with and without fluid viscosity. Figure \ref{fig:velo} displays the vertical component of the velocity, while Figure \ref{fig:mises} shows the von Mises stress, calculated from the total Cauchy stress tensor $\bsig - \alpha p \mathrm{I}_d$. The comparison between viscous ($\eta = 0.0015\, \unit{Pa\times s}$) and inviscid ($\eta = 0\, \unit{Pa\times s}$) cases reinforces our previous observations: the viscous model exhibits stronger attenuation of the slow compressional wave and the alteration of the fast wave propagation speed. These effects are clearly visible in both the velocity field and the stress distributions. We note that the implemented first-order absorbing boundary conditions do not provide perfect transparency, suggesting that a perfectly matched layer (PML) approach could yield more accurate results by reducing artificial wave reflections at the boundaries.

\begin{figure}[h!]
\begin{center}
\includegraphics[scale=0.3]{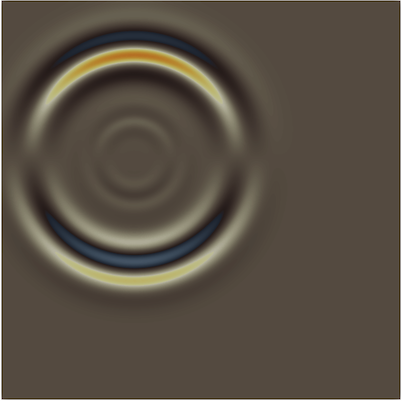}
\includegraphics[scale=0.3]{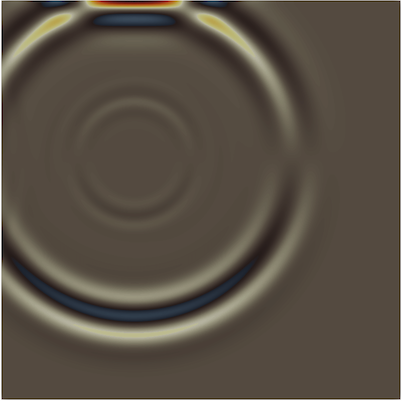}
\includegraphics[scale=0.3]{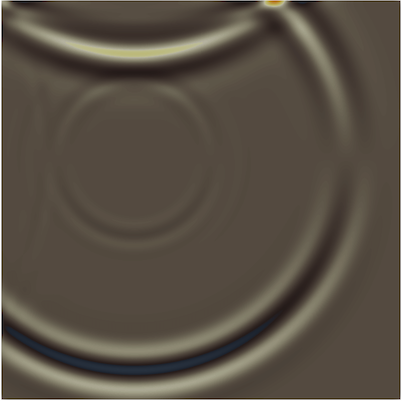}
\\
\includegraphics[scale=0.3]{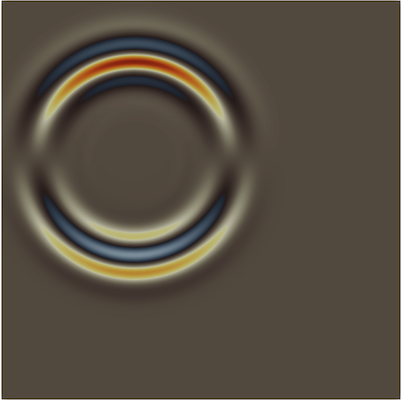}
\includegraphics[scale=0.3]{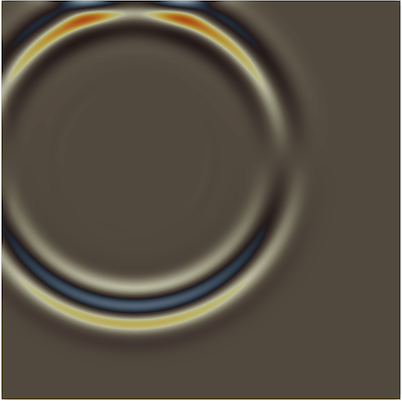}
\includegraphics[scale=0.3]{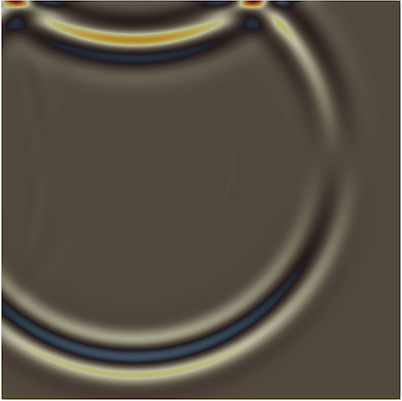}
\\
\includegraphics[width=0.5\textwidth]{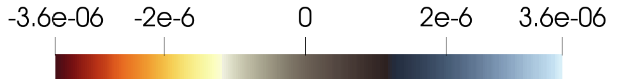}
\end{center}
\vspace{-4mm}
\caption{
Snapshots of the $y$-component of the solid velocity $\bu_s$ at times 0.7\,\unit{s}, 0.9\,\unit{s}, and 1.1\,\unit{s} (left to right panels). Problem \eqref{4field-a}-\eqref{4field-d} is solved using the source terms \eqref{sourceW}, boundary conditions \eqref{bcW}, and parameter set \eqref{coeffs}, with $h=100$, $k=5$, and $\Delta t = 0.005$. The top row shows results for an inviscid fluid ($\eta = 0$), while the bottom row corresponds to a viscous fluid ($\eta = 0.0015$).}
\label{fig:velo}
\end{figure}

\begin{figure}[h!]
\begin{center}
\includegraphics[scale=0.3]{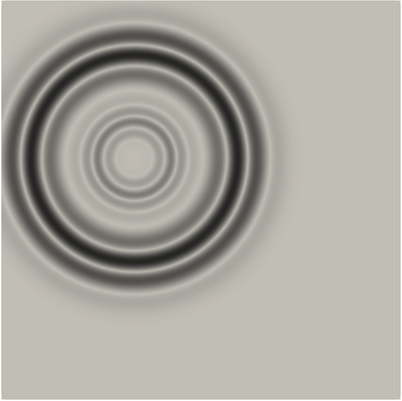}
\includegraphics[scale=0.3]{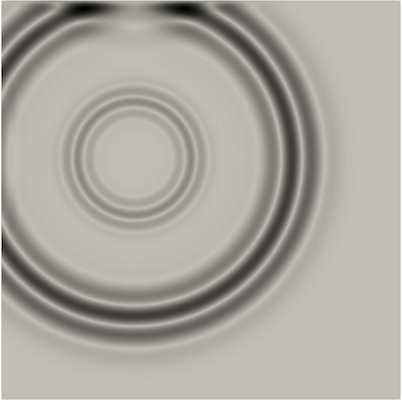}
\includegraphics[scale=0.3]{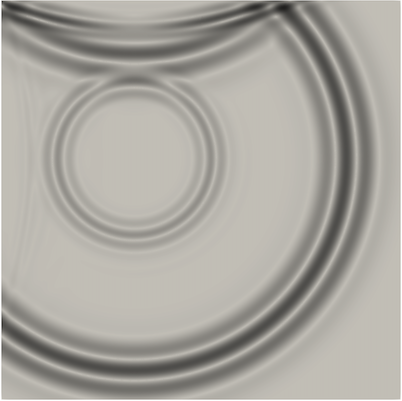}
\\
\includegraphics[scale=0.3]{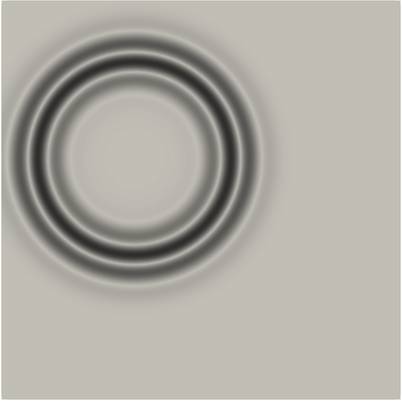}
\includegraphics[scale=0.3]{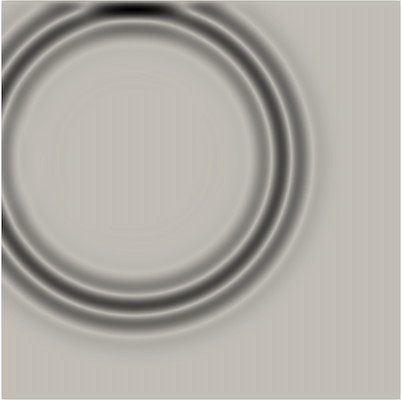}
\includegraphics[scale=0.3]{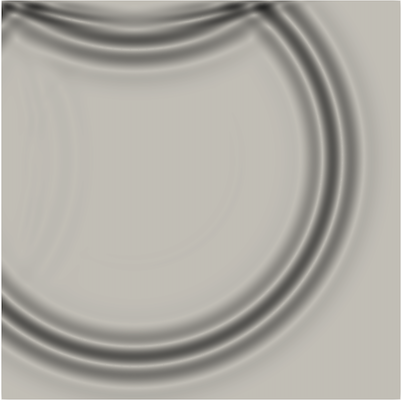}
\\
\includegraphics[width=0.5\textwidth]{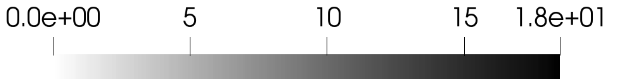}
\end{center}
\vspace{-4mm}
\caption{Snapshots of the von Mises stress, computed from the tensor $\bsig - \alpha p \mathrm{I}_d$, at times 0.7\,\unit{s}, 0.9\,\unit{s}, and 1.1\,\unit{s} (left to right panels). Problem \eqref{4field-a}-\eqref{4field-d} is solved using the source terms \eqref{sourceW}, boundary conditions \eqref{bcW}, and parameter set \eqref{coeffs}, with $h=100$, $k=5$, and $\Delta t = 0.005$. The top row shows results for an inviscid fluid ($\eta = 0$), while the bottom row corresponds to a viscous fluid ($\eta = 0.0015$).}
\label{fig:mises}
\end{figure}

\section{Conclusion}
This study presents a novel formulation of the Biot model for low-frequency wave propagation in poroelastic media. We establish well-posedness and energy stability of the variational formulation and introduce a hybridizable discontinuous Galerkin (HDG) space discretization method for both 2D and 3D problems. Our $hp$-convergence analysis of the semi-discrete scheme, coupled with Crank-Nicolson temporal discretization, provides a robust theoretical foundation. Numerical experiments confirm the effectiveness and accuracy of the method. Future work will extend this approach to poroviscoelastic media, incorporating viscous dissipation effects within the solid skeleton.

\bibliographystyle{plainnat}
\bibliography{cvBib.bib}

\end{document}